\documentclass[12pt]{amsart}
\usepackage{graphicx,amssymb,amscd}
\usepackage{enumitem}
\usepackage[usenames]{color}
\usepackage{amsthm,amsfonts,amsmath,amssymb,latexsym,epsfig,mathrsfs,yfonts,marvosym,latexsym,epsfig,tikz}
\usepackage{graphicx,xcolor}
\usepackage{url}
\usepackage[colorlinks=true]{hyperref}
\usepackage[utf8]{inputenc}
\usepackage[T1]{fontenc}
\usepackage[margin=0.835in]{geometry}

\AtEndDocument{\vfill\eject\batchmode}

\DeclareMathAlphabet\oldmathcal{OMS}        {cmsy}{b}{n}
\SetMathAlphabet    \oldmathcal{normal}{OMS}{cmsy}{m}{n}
\DeclareMathAlphabet\oldmathbcal{OMS}       {cmsy}{b}{n}
\usepackage[active]{srcltx}
\usepackage{eucal}

\newtheorem{theorem}{Theorem}[section]
\newtheorem{lemma}[theorem]{Lemma}
\newtheorem{proposition}[theorem]{Proposition}

\newtheorem{question}[theorem]{Question}
\newtheorem{def/prop}[theorem]{Definition/Proposition}

\theoremstyle{definition}
\newtheorem{definition}[theorem]{Definition}
\newtheorem{remark}[theorem]{Remark}
\newtheorem*{ack}{Acknowledgements}
\newtheorem{example}{Example}[section]

\DeclareSymbolFont{bbold}{U}{bbold}{m}{n}
\DeclareSymbolFontAlphabet{\mathbbold}{bbold}

\def\BOne{\mathchoice{\scalebox{1.16}{$\displaystyle\mathbbold 1$}}{\scalebox{1.16}{$\textstyle\mathbbold 1$}}{\scalebox{1.16}{$\scriptstyle\mathbbold 1$}}{\scalebox{1.16}{$\scriptscriptstyle\mathbbold 1$}}}
\def\fract#1#2{\raise4pt\hbox{$ #1 \atop #2 $}}

\def\bbc{{\mathbb C}}

\def\bbf{{\mathbb F}}

\def\bbp{{\mathbb P}}
\def\bbq{{\mathbb Q}}
\def\bbr{{\mathbb R}}
\def\bbs{{\mathbb S}}
\def\bbt{{\mathbb T}}

\def\bbz{{\mathbb Z}}

\def\gri{\iota}
\def\grk{\kappa}

\def\grD{\Delta}

\def\grS{\Sigma}

\def\ell{{\bf l}}

\def\bfw{{\bf w}}

\def\cala{{\mathcal A}}

\def\cald{{\mathcal D}}
\def\cale{{\mathcal E}}
\def\calf{{\mathcal F}}

\def\cali{{\mathcal I}}

\def\call{{\mathcal L}}
\def\calm{{\mathcal M}}

\def\calo{{\mathcal O}}

\def\calr{{\mathcal R}}
\def\cals{{\oldmathcal S}}

\def\calw{{\mathcal W}}

\def\calS{{\mathcal S}}

\def\gg{{\mathfrak g}}
\def\gh{{\mathfrak h}}

\def\gt{{\mathfrak t}}

\def\gz{{\mathfrak z}}

\def\gC{{\mathfrak C}}

\def\gR{{\mathfrak R}}

\def\<{\langle}
\def\>{\rangle}
\def\ra#1{\to}

\newcommand{\x}{r}

\def\fract#1#2{\raise4pt\hbox{$ #1 \atop #2 $}}
\def\decdnar#1{\phantom{\hbox{$\scriptstyle{#1}$}}
\left\downarrow\vbox{\vskip15pt\hbox{$\scriptstyle{#1}$}}\right.}

\def\lra{\longrightarrow}

\def\hook{\mathbin{\hbox to 6pt{%
                 \vrule height0.4pt width5pt depth0pt
                 \kern-.4pt
                 \vrule height6pt width0.4pt depth0pt\hss}}}

\begin{document}

\title{Robustness of CSC Sasaki existence under the join operation}

\author[Charles Boyer]{Charles P. Boyer}
\address{Charles P. Boyer, Department of Mathematics and Statistics,
University of New Mexico, Albuquerque, NM 87131.}
\email{cboyer@unm.edu} 
\author[Christina T{\o}nnesen-Friedman]{Christina W. T{\o}nnesen-Friedman}
\address{Christina W. T{\o}nnesen-Friedman, Department of Mathematics, Union
College, Schenectady, New York 12308, USA }
\email{tonnesec@union.edu}
\thanks{The first author was partially supported by a grant from the
Simons Foundation (\#519432).}
\date{\today}

\begin{abstract}
The main purpose of this work is to explore the existence of constant scalar curvature Sasaki metrics in the Sasaki cone of the join of two regular Sasaki manifolds, $M_1$ and $M_2$.
Furthermore, we consider some cases of continuous families of extremal Sasaki twins.

\end{abstract}

\maketitle
\vspace{-7mm}


\tableofcontents

\section*{Introduction}\label{intro}
The Sasaki join construction is the Sasaki analog of the product of K\"ahler structures on the base.
Let $M_1$ and $M_2$ be quasiregular Sasaki manifolds and consider the join $M_1\star_{l_1,l_2} M_2$ as defined in diagram \eqref{s2comdia} below. In \cite{BoTo19a} we describe the Sasaki join when one of the manifolds, say $M_2$, is the weighted 3-sphere $S^3_\bfw$.   The special case where $M_1$ is a circle bundle over a Riemann surface is treated in \cite{BoTo13,CaSt25}. We are interested in seeing how much information (topologically and geometrically) we can obtain in other cases by explicitly presenting examples. 

A reasonable question (Question \ref{generalquestion}) one might ask when considering a join $M_1\star_{l_1,l_2} M_2$ is the following: {\it Does the Sasaki cone of a join of two arbitrary Sasaki manifolds that each admit a cscS-ray (constant scalar curvature Sasaki-ray) somewhere in their 
individual Sasaki cones always have a (potentially irregular) cscS-ray?} In Section \ref{main}, we shall see that the general answer to this question is ``no,'' but we will also see many cases of $M_1$ and $M_2$ where 
$M_1\star_{l_1,l_2} M_2$  will have a cscS-ray. 

Related to the question above we could also ask the following: {\it Suppose $M_1$ or $M_2$ does NOT admit a cscS-ray. Might $M_1\star_{l_1,l_2} M_2$ still admit a cscS-ray?} 
In Example \ref{resurrection} we will see that this is indeed possible. An ample supply of examples showing this possibility also follows from Theorem B in the recent work \cite{ApLaPa26}.
Thus, we can both loose and gain cscS existence when we perform the join operation.

We begin with a brief introduction to Sasaki geometry in Section \ref{sasreview} and then continue in Section \ref{genjoin} with an introduction to the Sasaki join.
After this we have a transition section (Section \ref{transition}), where we explain how to (implicitly) study the existence of extremal and constant scalar curvature Sasaki metrics on certain joins by working with positive Killing potentials on 
product K\"ahler manifolds. This technique is due to the work by Apostolov and Calderbank \cite{ApCa18}.

In Sections \ref{main} and \ref{twins} we will (implicitly) consider various examples of joins $M_1\star_{l_1,l_2} M_2$ between two smooth Sasaki manifolds with focus on the 
$\gt^+(M_2)$ part of the Sasaki cone $\gt^+(M_1\star_\ell M_2)= (\gt^+(M_1)\oplus \gt^+(M_2))/\{L_\ell\}$ (see Equation \ref{t+add2}).
If $\dim \gt^+(M_1)=1$, this comprises all of $\gt^+(M_1\star_\ell M_2)$, but in general it is just a proper sub cone of the full unreduced Sasaki cone of the join.

\begin{itemize}
\item $M_1=S^3$ (which is an $S^1$-bundle over $\bbc\bbp^1$, so $M_1$ is regular) and $M_2$ is an $S^1$-bundle over the $k^{th}$ Hirzebruch surface $\bbf_k= \bbp(\calo \oplus \calo(k))\rightarrow \bbc\bbp^1$, which we take to be simply connected and again is regular.
By Proposition \ref{affirmative}, we have that any such join has at least one cscS ray in its Sasaki cone and in certain cases there will be at least three such rays. Topologically, $M_2$ is homeomorphic to $S^2\times S^3$ when $k$ is even and homeomorphic to the nontrivial $S^3$ bundle over $S^2$ when $k$ is odd. From this and Equation \eqref{lexacthom1} below we see that for all relatively prime pairs $\ell=(l_1,l_2)$ the join $S^3\star_\ell M_2$ is simply connected with $\pi_2(S^3\star_\ell M_2)=\bbz^2$.

\item In a similar vein $M_1$ is the compact Heisenberg manifold $\bbr^3/\Gamma$ , which can be represented as a nontrivial $S^1$-bundle over $T^2$ and $M_2$ is still a primitive $S^1$-bundle over the $k^{th}$ Hirzebruch surface $\bbf_k= \bbp(\calo \oplus \calo(k))\rightarrow \bbc\bbp^1$. Here $\Gamma$ is the integral Heisenberg group.
So $\pi_1(M_1)=\Gamma$ which is a nonsplit extension of $\bbz^2$ by $\bbz$, and $\pi_2(M_1)=0$, whereas $\pi_1(M_2)=\BOne$ and $\pi_2(M_2)=\bbz$.  By Proposition \ref{affirmative2} we have that any such join has at least one cscS ray in its Sasaki cone. 
\item $M_1=S^3$ and $M_2$ is some $S^1$-bundle over a ruled surface $\bbs_k$ of the form $\bbs_k=\bbp(\calo \oplus L_k)\rightarrow \Sigma_{\gg}$, 
where $\Sigma_{\gg}$ is any compact Riemann surface of genus $\gg>0$ and $L_k$ is a holomorphic line bundle of degree $k$ ($\int_{\Sigma_{\gg}} c_1(L_k)=k$).
While $\gt^+(M_1\star_\ell M_2)$ is not necessarily exhausted by extremal Sasaki rays in this case (See Part (2) of Example \ref{quasiregularex}), Proposition \ref{affirmative3} tells us that also in this case 
the Sasaki cone will contain at least one cscS ray.
\item $M_1$ is some $S^1$-bundle over $\Sigma_{\gg_1}$, where $\Sigma_{\gg_1}$ is a compact Riemann surface of genus $\gg_1>1$ and $M_2$ is an $S^1$-bundle over a ruled surface $\bbs_k$ of the form 
$\bbs_k=\bbp(\calo \oplus L_k)\rightarrow \Sigma_{\gg_2}$ where $\Sigma_{\gg_2}$ is a compact Riemann surface of genus $\gg_2>1$ and $L_k$ is a holomorphic line bundle of degree $k$.
From Examples \ref{nocscex} and \ref{endonposnote}, we discover sub-cases of this, where  $\gt^+(M_1\star_\ell M_2)$ has no extremal Sasaki rays whatsoever, as well as sub-cases where there is at least one cscS ray in the Sasaki cone.
The sub-cases are determined by the choices made of $\gg_2$ and $k$ as well as a polarization of $\Sigma_{\gg_1}\times \bbs_k$, which in turn determines the values of $l_1$, and $l_2$.
In Example \ref{resurrection} we will then iterate the regular join process by letting the regular Sasaki manifold resulting from the join in Example \ref{nocscex} play the role of one of the regular Sasaki manifolds in a new join with $S^3$. 
In other words we are considering $S^3 \star_{\tilde{l_1},\tilde{l_2}}\left(M_1\star_{l_1,l_2} M_2\right)$. 
\item $M_1$ is some $S^1$-bundle over $\bbc\bbp^2\#q\overline{\bbc\bbp}^2$,
where $4 \leq q \leq 8$ ($\bbc\bbp^2$ blown up at
$q$ generic points), and $M_2$ is an $S^1$-bundle over a ruled surface $\bbs_k$ of the form 
$\bbs_k=\bbp(\calo \oplus L_k)\rightarrow \Sigma_{\gg_2}$ where $\Sigma_{\gg_2}$ is a compact Riemann surface of genus $\gg_2>1$ and $L_k$ is a holomorphic line bundle of degree $k$.
For a particular choice of $\gg_2$ and $k$, Example \ref{cscmoatseparated} exhibits a choice of polarization of $\Sigma_{\gg_1}\times \bbs_k$ (i.e. a choice of $l_1$ and $l_2$) where the
Sasaki cone is exhausted by extremal Sasaki rays, of which three of them are cscS, as well as a choice of polarization of $\Sigma_{\gg_1}\times \bbs_k$ (i.e. a choice of $l_1$ and $l_2$) where
the Sasaki cone has two cscS rays which are in separate connected components of the set of extremal Sasaki rays in the Sasaki cone, The two connected components are separated by a ``moat'' in the Sasaki cone 
that has no extremal Sasaki metrics. See Figure \ref{moat}.
\item $M_1$ is some $S^1$-bundle over a compact Hodge manifold of negative scalar curvature and $M_2$ is the standard Sasaki manifold $S^{2n+1} \rightarrow \bbc\bbp^n$.
For a very special choice of $(l_1,l_2)$, Proposition \ref{manytwins} yields a case where there is a continuous family (sometimes exhaustive) of so-called extremal Sasaki twins in the Sasaki cone.
 \emph{Extremal Sasaki twins}  on a pseudo-convex CR structure $(\cald,J)$ of Sasaki type are extremal Sasaki structures all compatible with $(\cald,J)$ and having commuting non-colinear Sasaki-Reeb vector fields. 
 Usually these come in pairs (therefore the term ``twins''), but occasionally we get a continuous family. A well-known case of this is when the Sasaki manifold is $S^{2n+1}$ over $\bbc\bbp^n$ 
 and more generally for a regular Sasaki manifold over a Bochner flat K\"ahler manifold
 \cite{Bry01,ApCaGa06,ApCa18}. Proposition \ref{manytwins} includes cases where the regular transverse K\"ahler structure in not necessarily Bocher flat.
See Remark \ref{polarization} for more discussion about this case. 
\end{itemize}

\begin{ack}
The authors would like to thank Vestislav Apostolov for helpful conversations concerning the material in Section \ref{otherweights}.
\end{ack}

\section{Brief Review of Sasaki Geometry}\label{sasreview}

Recall that a Sasakian structure on a contact manifold $M^{2n+1}$ of dimension $2n+1$ is a special type of contact metric structure $\cals=(\xi,\eta,\Phi,g)$ with underlying almost CR structure $(\cald,J)$ where $\eta$ is a contact form such that $\cald=\ker\eta$, $\xi$ is its Reeb vector field, $J=\Phi |_\cald$, and $g=d\eta\circ (\BOne \times\Phi) +\eta\otimes\eta$ is a Riemannian metric. $\cals$ is a Sasakian structure if $\xi$ is a Killing vector field and the almost CR structure is integrable, i.e. $(\cald,J)$ is a CR structure. We refer to \cite{BG05} for the fundamentals of Sasaki geometry. We call $(\cald,J)$ a {\it CR structure of Sasaki type}, and $\cald$ a {\it contact structure of Sasaki type}. For convenience we shall always assume that the Sasaki manifolds $M^{2n+1}$ are compact and connected unless explicitly stated otherwise.

Given a Sasaki manifold $M$, if the Reeb vector field $\xi$ generates a locally free circle action, the quotient will be a projective algebraic orbifold which we write as the pair $(N,\grD_N)$ where $\grD_N$ is a sum of irreducible branch divisors, viz. 
$$\grD_N= \sum_{j=1}^k\bigl(1-\frac{1}{m_j}\bigr)D_j$$
where $k=\dim {\rm Div}(N)$, $m_j$ are the ramification indices, and $D_j\in {\rm Div}(N)$ the group of Weil divisors on $N$.

Within a fixed contact CR structure $(\cald,J)$ there is a conical family of Sasakian structures known as the Sasaki cone. We are also interested in a variation within this family. To describe the Sasaki cone we fix a Sasakian structure $\cals_o=(\xi_o,\eta_o,\Phi_o,g_o)$ on $M$ whose underlying CR structure is $(\cald,J)$ and let $\gt$ denote the Lie algebra of a maximal torus $\bbt$ in the automorphism group of $\cals_o$. The {\it (unreduced) Sasaki cone} \cite{BGS06} is defined by
\begin{equation}\label{sascone}
\gt^+(\cald,J)=\{\xi\in\gt~|~\eta_o(\xi)>0~\text{everywhere on $M$}\},
\end{equation}
which is a cone of dimension $k\geq 1$ in $\gt$. The reduced Sasaki cone $\grk(\cald,J)$ is $\gt^+(\cald,J)/\calw$ where $\calw$ is the Weyl group of the maximal compact subgroup of the CR automorphism group $\gC\gR(\cald,J)$. Indeed, $\grk(\cald,J)$ is the moduli space of Sasakian structures with underlying CR structure $(\cald,J)$. However, it is more convenient to work with the unreduced Sasaki cone $\gt^+(\cald,J)$. It is also clear from the definition that $\gt^+(\cald,J)$ is a cone under the transverse scaling defined by
\begin{equation}\label{transscale}
\cals=(\xi,\eta,\Phi,g)\mapsto \cals_a=(a^{-1}\xi,a\eta,g_a),\quad g_a=ag+(a^2-a)\eta\otimes\eta, \quad a\in\bbr^+
\end{equation}
So $\gt^+(\cald,J)$ is a cone and since the Reeb vector field $\xi$ is a Killing vector field, we have $\dim\gt^+(\cald,J)\geq 1$. Moreover, it follows from contact geometry that $\dim\gt^+(\cald,J)\leq n+1$. When $\dim\gt^+(\cald,J)=n+1$ we have a toric contact manifold of Reeb type studied in \cite{BM93,BG00b,Ler02a,Ler04,Leg10,Leg16}. In this case there is a strong connection between the geometry and topology of $(M,\cals)$ and the combinatorics of $\gt^+(\cald,J)$. Much can also be said in the complexity 1 case ($\dim\gt^+(\cald,J)=n$) \cite{AlHa06}.
 
We often have need to deform the contact structure $\cald\mapsto \cald_\varphi$ by a contact isotopy $\eta\mapsto \eta +d^c\varphi$ where $\varphi\in C^\infty(M)^\bbt$ is a smooth function invariant under the torus $\bbt$. We note that the Sasaki cone $\gt^+(\cald,J)$ is invariant under such contact isotopies, that is, $\gt^+(\cald_\varphi,J_\varphi)= \gt^+(\cald,J)$. For each $\varphi\in C^\infty(M)^\bbt$, $\cald_\varphi\longrightarrow TM$ gives a splitting of the exact sequence
$$0\longrightarrow L_{\xi_o}\longrightarrow TM\longrightarrow Q\longrightarrow 0$$
with $J_\varphi=\Phi |_{{\cald}_{\varphi}}$.  Furthermore, each choice of Reeb vector field $\xi\in\gt^+(\cald,J)$ gives rise to an infinite dimensional contractible space $\calS(M,\xi)$ of Sasakian structures \cite{BG05}. We shall often make such a choice $\cals=(\xi,\eta,\Phi,g)\in\calS(M,\xi)$ and identify it with the element $\xi\in\gt^+(\cald,J)$.

\begin{definition}\label{sewz.1}We denote by $\cals\calo$ the groupoid whose object set is the set of
compact quasi-regular Sasakian orbifolds and whose morphisms are orbifold diffeomorphisms, by $\cals\calm$ the
subgroupoid of $\cals\calo$ that are smooth manifolds and smooth diffeomorphisms, and by $\calr\subset
\cals\calm$ the subset of compact, simply connected, regular Sasakian manifolds.
\end{definition}

The groupoids $\cals\calo$ and $\cals\calm$ are graded by dimension
$$\cals\calo =\bigoplus_{n=0}^\infty \cals\calo_{2n+1}, \qquad \qquad  \cals\calm =\bigoplus_{n=0}^\infty \cals\calm_{2n+1}$$ 
and for each pair of relatively prime positive integers $(l_1,l_2)$, we can define a join called the $\ell$-{\it join} which gives a graded multiplication
\begin{equation}\label{joinmaps}
\star_{l_1,l_2}:\cals\calm_{2n_1+1}\times \cals\calm_{2n_2+1}\ra{1.5} \cals\calo_{2(n_1+n_2)+1}
\end{equation}
which we now describe.

\section{The General Join Construction}\label{genjoin}
The join construction is the Sasaki analogue for products in K\"ahler geometry. It was first described in the context of Sasaki-Einstein manifolds \cite{BG00a}, but then developed more generally in \cite{BGO06}, see also Section 7.6.2. of \cite{BG05}. Its relation to the de Rham decomposition Theorem and reducibility questions \cite{HeSu12b} were studied in \cite{BHLT16}. Joins of the form $M\star_\ell S^3_\bfw$ were considered in full generality in \cite{BoTo19a}. Here we briefly recall the general join $M_1\star_{l_1,l_2} M_2$ from \cite{BGO06} and \cite{BG05} where $M_i$ are smooth quasiregular Sasaki manifolds. Generally,  the join $M_1\star_{l_1,l_2} M_2$ will be an orbifold.  Lemma \ref{joinorderlem} below gives the conditions for $M_1\star_{l_1,l_2} M_2$ to be a smooth manifold. Given two manifolds $M_1,M_2$ with locally free circle actions $\cala_1,\cala_2$, respectively, and a pair of relatively prime positive integers $\ell=(l_1,l_2)$ we define the $\ell$-join (or just join) of $M_1$ and $M_2$ as the quotient $M_{\ell}=M_1\star_{l_1,l_2} M_2$ of the fibration
\begin{equation}\label{Sjoineqn}
S^1\longrightarrow M_1\times M_2\longrightarrow M_1\star_{l_1,l_2} M_2
\end{equation}
where the $S^1$ action is given by 
\begin{equation}\label{S1action}
(x_1;x_2)\mapsto (\cala_1(l_2\theta,x_1);\cala_2(-l_1\theta,x_2)).
\end{equation}
The join $M_{\ell}=M_1\star_{l_1,l_2} M_2$ is then constructed from the following  commutative diagram
\begin{equation}\label{s2comdia}
\begin{matrix}  M_1\times M_2 &&& \\
                          &\searrow\pi_L && \\
                          \decdnar{\pi_{1}\times\pi_{2}} && M_1\star_{l_1,l_2} M_2 &\\
                          &\swarrow\pi && \\
                           (N_1,\grD_{N_1})\times(N_2,\grD_{N_2}) &&& 
\end{matrix} 
\end{equation}
where the $\pi_L$ is the projection generated by the circle action \eqref{S1action} which in turn is generated by the vector field 
\begin{equation}\label{Lvec}
L_\ell =\frac{1}{2l_1}\xi_{M_1}-\frac{1}{2l_2}\xi_{M_2}, 
\end{equation}
where $\xi_{M_i}$ is the Reeb vector field of the Sasaki manifold $M_i$ with contact form $\eta_{M_i}$. Note that the southwest arrow $\pi$ in diagram \eqref{s2comdia} is just the Boothby-Wang bundle over the product.
The quotient manifold $M_1\star_{l_1,l_2} M_2$ has a naturally induced quasi-regular Sasakian structure with contact 1-form 
\begin{equation}\label{conform}
\eta_\ell =l_1\eta_{M_1}+l_2\eta_{M_2}
\end{equation} 
and Reeb vector field
\begin{equation}\label{Reebjoin}
\xi_{\ell}=\frac{1}{2l_1}\xi_{M_1}+\frac{1}{2l_2}\xi_{M_2}.
\end{equation}

On the join $M_{\ell}=M_1\star_\ell M_2$ there is a direct sum decomposition of the tangent bundle \cite{HeSu12b, BHLT16}
\begin{equation}\label{tbunsplit}
TM_{\ell}=\cald_1\oplus \cald_2\oplus \{{\xi_\ell}\}
\end{equation}
which for $i=1,2$ gives rise to two foliations $\cale_i=\cald_i\oplus \{\xi_\ell\}$ of $M_{\ell}$ whose leaves are totally geodesic Sasaki submanifolds isomorphic, up to transverse scaling, as Sasakian structures to $M_1$ and $M_2$, respectively. Note that the intersection $\cale_1\cap \cale_2$ is just $\calf_{\xi_\ell}$, the foliation generated by the Reeb vector field $\xi_\ell$, and both the transverse metric $g^T$ and the contact bundle $\cald_{\ell}=\ker\eta_{\ell}$ split as direct sums. The natural numbers $l_1,l_2$ are generally contact invariants. Since the CR-structure $(\cald_{\ell},J)$ is the horizontal lift of  the complex structure on $N_1\times N_2$, this splits as well.  

In the following we are going to present some interesting cases involving the existence of constant scalar curvature Sasaki metrics or more generally extremal Sasaki metrics in the Sasaki cone of certain joins between two smooth regular Sasaki manifolds; hence, we have a special interest in the regular case when \eqref{s2comdia} simplifies to
\begin{equation}\label{s2comdiasmooth}
\begin{matrix}  M_1\times M_2 &&& \\
                          &\searrow\pi_L && \\
                          \decdnar{\pi_{1}\times\pi_{2}} && M_1\star_{l_1,l_2} M_2 &\\
                          &\swarrow\pi && \\
                          N_1\times N_2 &&& 
\end{matrix} 
\end{equation}
and $\pi: M_1\star_{l_1,l_2} M_2  \rightarrow N_1\times N_2$ is a regular Boothby-Wang construction. In this case, as seen later, we will choose to explore the Sasaki cone via positive Killing potentials on the
K\"ahler manifold $N= N_1\times N_2$.

We are interested in the case when the image of the map \eqref{joinmaps} is a smooth manifold. Recall Definition 7.1.1 in \cite{BG05} that the order $\Upsilon$ of a quasiregular Sasaki manifold is the lcm of the orders of the leaf holonomy groups of the characteristic foliation. Moreover, $\Upsilon$ coincides with the lcm of the order of the local uniformizing groups of the quotient orbifold $N$.

From Proposition 7.6.6 in \cite{BG05} we have

\begin{lemma}\label{joinorderlem}
Let $M_i$ be smooth quasiregular Sasaki manifolds of order $\Upsilon_i$. Then the $\ell$-join $M_{\ell}=M_1\star_\ell M_2$ is smooth if and only if $\gcd(\Upsilon_1l_2,\Upsilon_2l_1)=1$ in which case the order of $M_\ell$ is $\Upsilon_1\Upsilon_2$. In particular, if the Sasakian structures on both $M_1$ and $M_2$ are regular, the join will be smooth for all pairs of relatively prime integers $l_1$ and $l_2$.
\end{lemma}

It is the latter case that concerns us most in this paper. In any case we easily obtain some topological information.
From the fibration \eqref{Sjoineqn}, we obtain the long exact homotopy sequence 
\begin{equation}\label{lexacthom1}
0\lra \pi_2(M_1)\oplus \pi_2(M_2)\lra \pi_2(M_1\star_\ell M_2)\fract{g}{\lra} \bbz\fract{f}{\lra} \pi_1(M_1)\oplus \pi_1(M_2)\lra \pi_1(M_1\star_\ell M_2)\lra \BOne
\end{equation}
and for $n>2$  
\begin{equation}\label{lexacthom2}
\pi_n(M_1\star_\ell M_2)=\pi_n(M_1)\oplus \pi_n(M_2).
\end{equation}
In particular, equation \eqref{lexacthom1} implies that $\ker{f}\approx {\rm im}(g)\approx k\bbz$ for some $k\in \bbz$.
A particular case of interest is 
\begin{example}\label{ex1}
Let $M_1=S^{2k+1}$ for $k\in \bbz^{+}$ and $M_2$ be the total space of an $S^1$ bundle over a compact Riemann surface $\grS_\gg$ of genus $\gg>0$. Then $\pi_2(M_1)=\pi_2(M_2)=\pi_1(M_1)=0$ and $\pi_1(M_2)$ is a nonsplit extension of $\pi_1(\grS_\gg)$ by $\bbz$. So $\pi_2(M_1\star_\ell M_2)\approx \bbz$ and the map $g$ in \eqref{lexacthom1} is multiplication by some $j\in \bbz$. So \eqref{lexacthom1} becomes
$$
\begin{CD}
 @. @. @. 0 \\
 @. @. @.  @VVV \\
 @. @. @. \bbz  \\
 @. @. @.  @VVV \\
 0 @>>>\pi_2(S^{2k+1}\star_\ell M_2)@>\times j>> \bbz @>f>> \pi_1(M_2)@>>> \pi_1(S^{2k+1}\star_\ell M_2) @>>> \BOne. \\
 @. @. @.  @VVV \\
 @. @. @. \pi_1(\grS_\gg) \\
@. @. @.  @VVV \\
 @. @. @. \BOne
\end{CD} 
$$
Note that $f$ defines an injection of $\bbz/j\bbz$ into $\pi_1(M_2)$ giving the short exact sequence
\begin{equation}\label{shexseq}
0\lra \bbz/j\bbz \lra \pi_1(M_2) \lra \pi_1(S^{2k+1}\star_\ell M_2)\lra \BOne.
\end{equation}
\end{example}

\subsection{The Sasaki Cone $\gt^+(M_1\star_\ell M_2)$ of the Join}
We begin by describing a maximal Abelian Lie algebra of the join $M_1\star_\ell M_2$. The Lie algebra $\gt_M$ of a quasi-regular Sasaki manifold $(M,\cals)$ can be written as $\gt_M=\{\xi_M\} +\gh_M$ where $\gh_M$ is the horizontal lift to $M$ of a maximal commutative Lie algebra of Hamiltonian vector fields on $N$. On $M_1\times M_2$ we have the Lie algebra $\gt_{M_1}\oplus \gt_{M_2}$ which induces the Lie algebra isomorphism
\begin{equation}\label{t+add}
\gt(M_1\star_\ell M_2)\approx (\gt_{M_1}\oplus \gt_{M_2})/\{L_{\ell}\}
\end{equation}
on $M_1\star_\ell M_2$ where $\{L_{\ell}\}$ denotes the one dimensional ideal generated by the vector field $L_{\ell}$ of Equation \eqref{Lvec}. It then follows from \eqref{Lvec} and \eqref{Reebjoin} that a maximal Abelian Lie algebra $\gt(M_1\star_\ell M_2)$ of the join satisfies

\begin{lemma}\label{torjoinlem}
There is an isomorphism of Abelian Lie algebras
$$\gt(M_1\star_\ell M_2)\approx \{\xi_{\ell}\} + \gh_{M_1} +\gh_{M_2}$$
\end{lemma}

Given the Sasaki cones $\gt^+(M_i)$ of the Sasaki manifolds $(M_i,\cals_i)$ we obtain the Sasaki cone of a join as an equality of conical sets
\begin{equation}\label{t+add2}
\gt^+(M_1\star_\ell M_2)= (\gt^+(M_1)\oplus \gt^+(M_2))/\{L_\ell\}.
\end{equation}
Clearly, we have
$$ \dim\gt^+(M_1\star_\ell M_2)= \dim\gt^+(M_1)+\dim \gt^+(M_2)-1. $$
Note that the right hand side of equation \eqref{t+add} says that the equivalence class always has a positive representative, that is it can be written as the sum of two vector fields each of which are pullbacks of Reeb vector fields from $M_1$ and $M_2$.

As in Section 2.3 of \cite{BoTo19a} there are maps $\gri_1:\gt^+(M_1)\longrightarrow \gt^+(M_1\star_\ell M_2)$ and  $\gri_2:\gt^+(M_2)\longrightarrow \gt^+(M_1\star_\ell M_2)$ defined by $\gri_i(a\xi_i+h_i)= a\xi_\ell +\bar{h_i}$ where $\bar{h_i}$ is the lift of a Hamiltonian vector field $h_i$ on $N_i$ by the map $\pi$ of diagram \eqref{s2comdia}. Here of course $a>0$. We call the image of $\gri_i$ in $\gt^+(M_1\star_\ell M_2)$ the $M_i$-subcone and by abuse of notation write $\gt^+(M_i)$. These two subcones intersect along the ray generated by $\xi_{\ell}$. The Sasakian structures that are most accessible through this construction are the elements of these subcones. Of course, the case $M_2=S^3$ is that studied in \cite{BoTo14a,BoTo19a}. Nevertheless, it is easy to show that

\begin{theorem}\label{orbquot}
Let $M_i$ be quasi-regular Sasakian manifolds and let $M_{\ell}=M_1\star_\ell M_2$ be their $\ell$-join. Then
the quotient of $M_{\ell}$ by the $S^1$ action generated by a quasi-regular Reeb vector field $\xi_i$ in the subcone $\gt^+(M_i)$  is an orbibundle with generic fibers $N_{\xi_i}/\bbz_m$ for some finite cyclic subgroup  $\bbz_m$ of $S^1$, where $N_{\xi_i}$ is the quotient of $M_i$ with respect to the circle generated by the Reeb vector field $\xi_i$. 
\end{theorem}

\begin{proof}
We let $\xi_i$ be a quasiregular Reeb vector field with quotient orbifold $(N_{\xi_i},\grD_{N_{\xi_i}})$. Suppose also that $\xi_i$ has period $m$ and lies in the subcone $\gt^+(M_i)$. Let $\xi_{i,m}$ denote the corresponding vector field with period $1$. Then the quotient of $M_i$ by the circle action generated by $\xi_{i,m}$ is $N_{\xi_i}/\bbz_m$. So when we define the $\ell$-join as in diagram \eqref{s2comdia} giving $M_{\ell}$ as an $S^1$ quotient of the product $M_1\times M_2$, we obtain an orbibundle with generic fibers $N_{\xi_i}/\bbz_m$.

\end{proof}

\section{Transition to Examples}\label{transition}

Recall that for a given $p\in {\mathbb R}$, a Kähler manifold $(N^n,J,g,\omega)$ with a fixed Killing potential $f$ (i.e $\call_{\nabla^gf}J=0$) is said to be an {\it $(f,p)$-extremal K\"ahler structure} if the $(f,p)$-\emph{scalar curvature} 
\begin{equation}\label{weightedscal}
Scal_{f,p}(g)= f^2 Scal(g)  -2(p-1) f\Delta_g f - p(p-1)|df|^2_g,
\end{equation} is a Killing potential, \cite{ApCa18}.  This is a direct extension of the now classical extremal K\"ahler metrics of Calabi \cite{cal82} which correspond to the case $(f,p)=(1,p)$ for any $p$. 
When $p=2n$, $Scal_{f,2n}(g)$ computes the scalar curvature of the Hermitian metric $h = f^{-2} g$. In particular, the case $n=2$ and $p=4$ is important due
to the discovery by LeBrun implying that a constant $Scal_{f,4}(g)$ corresponds to strongly Hermitian solutions of the Einstein-Maxwell equations \cite{Leb10,Leb15}.
On the trivial and first Hirzebruch surfaces, LeBrun found many such solutions (see \cite{Leb15, Leb16}) among which the Page metric \cite{Pag} on the first Hirzebruch surface is a special case.

Here our focus will be on the case where $p=n+2$ and $(N,J,g,\omega)$ is a compact K\"ahler manifold such that the K\"ahler class $[\omega]$ is rational. Then, after an appropriate rescaling,
$(J,g,\omega)$ is the transverse K\"ahler structure of a $(2n+1)$-dimensional Sasaki manifold $(M, \cald, J,\xi_1)$ in the sense that $\pi^*\omega=d\eta$ where $\eta$ is a contact $1$-form with regular Reeb vector field $\xi_1$ and such that $(\cald=\ker \eta, J)$ is a CR structure and $\pi : M\rightarrow N$ is the principal $S^1$-bundle defined by the regular foliation $\calf_{\xi_1}$. Any positive, smooth Killing potential $f$ on $(N,J,g,\omega)$ pulls back to a
$\xi_1$-invariant positive Killing potential $f$ on $(M, \cald, J,\xi_1)$ which in turn determines a new Sasaki structure $(M, \cald, J,\xi_f)$ with the same underlying CR-structure but another Reeb 
vector field, $\xi_f$ (with contact form $\eta_f=f^{-1}\eta$) such that $\xi_f$ and $\xi_1$ commute. For general references see Sections 8.1 and 8.2 of \cite{BG05}, Lemma 7.1 of \cite{BoTo11}, and Section 1 of  \cite{ApCa18}.
We now have the following lemma as obtained by Apostolov and Calderbank:

\begin{lemma}[\cite{ApCa18}]\label{ac}
The metric $g$ is an {\it $(f,n+2)$-extremal K\"ahler metric} if and only if $\xi_f$ is the Reeb vector field of an extremal Sasaki structure in the sense of \cite{BGS06}. 
\end{lemma}

Furthermore, from the work in \cite{ApCa18}, it follows that this extremal Sasaki metric has constant scalar curvature exactly when 
\begin{equation}\label{cscS}
Scal_{f,m+2}(g)/f = \text{constant}.
\end{equation}

We are now going to consider K\"ahler products of the form 
$(N,g,\omega)=(N_1\times N_2, g_1+g_2, \omega_1+\omega_2)$. Implicitly, we are assuming that
the K\"ahler class $\left[\omega/2\pi\right]$ is a rational rescale of an integer K\"ahler class on $N$. 
In that case we can do an overall rational rescale to get a primitive polarization of the complex manifold $N$
(corresponding to choosing primitive polarizations of the complex manifolds $N_1$ and $N_2$ and then co-prime positive integers $(l_1,l_2)$ for the join \eqref{s2comdiasmooth})
and hence have a natural Sasaki structure above this product (via the Boothby-Wang construction). 

We first remind ourselves of the following:
Assume $(N_1,g_1)$ and $(N_2,g_2)$ are two Riemannian manifolds and
let $g=g_1+g_2$ be the product metric on the product manifold
$N=N_1\times N_2$. Then we have a few elementary facts:
\begin{itemize}
\item The scalar curvature, $Scal_g$, of $g$ is simply given by the sum of the
(pull-backs) of the scalar curvatures of $g_1$ and $g_2$:
$Scal_g=Scal_{g_1}+Scal_{g_2}$.
\item The Laplacian, $\Delta_g$, of $g$ on $0$-forms is also just the sum of the Laplacians 
of $g_1$ and $g_2$:
$\Delta_g=\Delta_{g_1}+\Delta_{g_2}$.
\item If $f$ is a smooth real valued function on $N_i$ for $i=1$ or $i=2$, then the pull-back to $N$ of $f$ via the natural projection
is a smooth real valued function on $N$. Denoting for convenience this function by $f$ as well, we have that
(suppressing the obvious pull-backs in the notation)
$|df|^2_g=|df|^2_{g_i}$ and, by the above bullet point, also $\Delta_g f=\Delta_{g_i} f$.
\end{itemize}

\section{Robustness of the existence of cscS and extremal Sasaki structures under the join construction}\label{main}
The motivating question for the following examples was as follows:
\begin{question}\label{generalquestion}
Does the Sasaki cone of a join of two arbitrary Sasaki manifolds that each admit a cscS-ray somewhere in their 
individual Sasaki cones always have a (potentially irregular) cscS-ray?
\end{question}

We will see that the answer in general is no (Example \ref{nocscex}), but that there is a nice family of cases where the answer is yes (Propositions \ref{affirmative}, \ref{affirmative2}, and \ref{affirmative3}). 
Along the way, we will also discover new cases of multiple cscS-rays in a Sasaki cone and even a case where 
two cscS-rays are separated by a `moat' in the Sasaki cone in which not even extremal Sasaki rays exists (Example \ref{cscmoatseparated}).

\begin{remark}\label{cscSinMi}
Before we proceed, we remind the reader about the following facts:
\begin{itemize}
\item If $M_1$ is the Boothby-Wang constructed Sasaki manifold over a polarized compact K\"ahler manifold with constant scalar curvature, then obviously the Sasaki cone
admits a cscS-ray. (Namely, the regular ray.)
\item If $M_2$ is the Boothby-Wang constructed Sasaki manifold over a polarized minimal compact ruled surface $\bbp(\calo \oplus L)\rightarrow \Sigma$, where $\Sigma$ is a compact Riemann surface, 
$\calo$ is the trivial holomorphic line bundle over $\Sigma$, and $L$ is any holomorphic line bundle over $\Sigma$, then
the Sasaki cone also admits a cscS-ray. For the genus of $\Sigma$ at most one, this can be seen most directly from e.g. Theorem 3.1 in \cite{BHLT23}.   
Generally, for an arbitrary value of the genus of $\Sigma$, the $m_1=m_2=1$ case of Proposition 4.22 of \cite{BHLT16} tells us that $M_2$ can also be realized as the regular Sasaki structure in a $S^3_\bfw$-join 
as defined in \cite{BoTo14a}. 
Then the existence of a cscS-ray follows from  Theorem 1.1 of \cite{BoTo14a}. For an alternative approach, see \cite{BoTo25}.
\end{itemize}
\end{remark}

\subsection{The case of $N_1\times N_2=\bbc\bbp^1 \times \bbf_k$}
One of the Fano threefolds appearing in \cite{ACCF23} is the product $\bbc\bbp^1 \times \bbf_1$, where $\bbf_1$ is the first Hirzebruch surface. 
We know that this K\"ahler manifold does not admit a K\"ahler-Einstein metric\footnote{In fact, every K\"ahler class on $\bbc\bbp^1 \times \bbf_1$ (which by the K\"unneth formula will be a product class) admits a product K\"ahler metric that is extremal 
but not of constant scalar curvature.}. However, since the natural Sasaki manifold arising as the circle bundle over $\bbc\bbp^1 \times \bbf_1$ defined by the canonical polarization is clearly a toric Sasaki manifold, we know 
from \cite{FOW06} that the Sasaki cone admits a Sasaki-Einstein metric up to isotopy.\footnote{In this section everything will be ``up to isotopy.'' In particular, when we consider a Boothby-Wang construction given by some polarization, that is, some choice of an integer K\"ahler class, we simply pick a random K\"ahler form in that class for the actual construction. Another choice of K\"ahler form in the same class would give us a Sasaki structure in the same isotopy class.} 
Thus, this is a case that would support an affirmative answer to Question \ref{generalquestion}. 

More generally we can ask
 \begin{question}\label{question2}
Choosing an arbitrary polarization of $N_k= \bbc\bbp^1 \times \bbf_k$, where $\bbf_k$ is the $k^{th}$ Hirzebruch surface, will the corresponding Sasaki manifold over $N_k$ admit a Sasaki structure with constant scalar curvature 
somewhere in the Sasaki cone?
\end{question}

\begin{remark}
Concerning the manifold $N_k$ we note the following:
\begin{itemize}
\item The manifold $N_k$ is a special case of a stage $3$ Bott manifold as studied in \cite{BoCaTo17}. Thus, answering Question \ref{question2} can be viewed as a companion result to the work in \cite{BoTo21}.
\item The Sasaki manifolds over $N_k$ are each of the form $M_1\star_{l_1,l_2} M_2$ with $M_1=S^3$. That is, they are all $S^3$-joins 
with a Sasaki manifold $M_2$ over $\bbf_k$. 
In previous work by the authors the natural sub-cone of the Sasaki cone (of the join) arising from the Sasaki cone of $S^3$ was the object of interest for such joins.
In this paper we are looking in a different direction of the Sasaki cone. The main reason for this is that $\bbf_k$ does not admit a constant scalar curvature K\"ahler metric and, in particular, 
the joins we are considering in this paper are not of the same nature as the joins considered in e.g. Theorem 3.2 of \cite{BoTo19a}.
\item We might think of $N_k$ as a sort of degenerate admissible manifold and hence Theorem 3.1 in \cite{BHLT23} would suggest a positive answer to Question \ref{question2}.
We prove in Section \ref{answerquestion2} that this is indeed true.
\end{itemize}

\end{remark}

\subsubsection{Review of admissible metrics on $\bbf_k$}\label{admreview}
In the following we will assume that $k>0$. This is okay since Question \ref{question2} is completely trivial when $k=0$.
Then we define the $k^{th}$ Hirzebruch surface, $\bbf_k$, to be (the total space of) the bundle $\bbp(\calo \oplus \calo(k))\rightarrow \bbc\bbp^1$, where $\calo(k)$ denotes the degree $k$ holomorphic line bundle over $\bbc\bbp^1$ and $\calo$ denotes the trivial bundle over $\bbc\bbp^1$.
In \cite{cal82} Calabi famously constructed extremal K\"ahler metrics in every K\"ahler class of $\bbf_k$ and these metrics are each a special case of \emph{admissible K\"ahler metrics} as defined in \cite{ACGT08}. 
What follows is a quick overview (very similar to Section 3.1 of \cite{BHLT25}) on how to build admissible K\"ahler metrics
on $\bbf_k$. We will use the notation from \cite{ACGT08} and \cite{ApMaTF18} to which we refer for references as well as more technical details on what follows below.

Let $g_{{\mathbb C}{\mathbb P}^1}$ be the
K\"ahler metric on ${\mathbb C}{\mathbb P}^1$ of constant scalar curvature $2s$, with K\"ahler form
$\omega_{{\mathbb C}{\mathbb P}^1}$, such that
$c_{1}({\mathcal O}(k)) =\left[\frac{\omega_{{\mathbb C}{\mathbb P}^1}}{2 \pi}\right]$. Then $g_{{\mathbb C}{\mathbb P}^1}$ is a certain rescale of the Fubini-Study metric and,
since the Ricci form $\rho_ {{\mathbb C}{\mathbb P}^1}$ equals $s\omega_{{\mathbb C}{\mathbb P}^1}$, we see that $s= 2/k$.

There exist Hermitian metrics $h_0$ on $\calo$ and $h_{\infty}$ on $\calo(k)$ whose  respective Chern connections  have curvatures
$0$ and  $\omega_{\bbc \bbp^1}$, respectively. Let $r_0$ and $r_{\infty}$ denote the corresponding fibre-wise norm functions.
We denote the generator of the circle action on $\calo$  by $K_{0}$  and the generator of the circle action on $\calo(k)$
by $K_{\infty}$. Using the Chern connections of $(\calo,  h_0)$  and $(\calo(k),h_{\infty})$, we let $\theta_0$ and $\theta_{\infty}$ be the connection 1-forms defined on the corresponding unitary bundles, i.e. satisfying
\begin{equation*}
\begin{split}
 \theta_0(K_0) &=1,  \ \ d\theta_0 = 0; \\
\theta_{\infty} (K_{\infty})&=1,  \ \ d\theta_{\infty} = -\omega_{\bbc \bbp^1}.
 \end{split}
 \end{equation*}
Thus, the fibres-wise Euclidean structures (viewed as tensors on the total spaces of $\calo$ and $\calo(k)$) take the following (angular momentum) form
\begin{equation*}
g_0=\frac{{d\gz_0}\otimes d\gz_0}{2\gz_0} + 2\gz_0 ( \theta_0\otimes \theta_0), \ \ g_{\infty}=\frac{d\gz_{\infty}\otimes d\gz_{\infty}}{2\gz_{\infty}} + 2\gz_{\infty} (\theta_{\infty}\otimes \theta_{\infty}),
\end{equation*}
where  $\gz_0:= r^2_0/2$, $\gz_{\infty} := r_{\infty}^2/2$ are the fibre-wise momentum coordinates.

Let $0<x<1$  be a fixed real number. We then  consider the smooth positive semidefinite tensor  on the total space of $\calo \oplus \calo(k)$:
\begin{equation*}
\frac{(1+x)\gz_0+ (1-x)\gz_{\infty}}{2x}g_{\bbc \bbp^1} + g_{0} + g_{\infty}.
\end{equation*}
Considering the ``K\"ahler quotient''  for this tensor with respect to the $S^1$-action generated by $K_0 + K_{\infty}$ at the  level set $\gz_0 + \gz_{\infty}=2$ on $\calo \oplus \calo(k)$, we denote by  $g_c$ the smooth (possibly degenerate) tensor field  induced on 
$\bbf_k$ and by $\omega_x = g_c J_c$ the corresponding smooth $(1,1)$-form, where $J_c$ is the induced (canonical) complex structure on $\bbf_k$. Letting $\gz:=(\gz_0-\gz_{\infty})/2\in [-1,1]$, $(g_c, \omega_x)$ is written  on $\bbf_k^0:=\gz^{-1}(-1,1)$ as:
\begin{equation}\label{g}
g_c=\frac{1+x\gz}{x}g_{\bbc \bbp^1}+\frac {d\gz^2}
{\Theta_c (\gz)}+\Theta_c (\gz)\theta^2,\quad
\omega_x = \frac{1+x \gz}{x}\omega_{\bbc \bbp^1} + d\gz \wedge
\theta,
\end{equation}
where $\Theta_c (\gz)= 1-\gz^2$ and $\theta = {\theta}_{0}- {\theta}_{\infty}$ satisfies
\begin{equation}\label{theta}
d\theta = \omega_{\bbc \bbp^1}.
\end{equation}
It follows that $(g_c, \omega_x)$ defines a K\"ahler metric on $\bbf_k^0$ and it is shown in \cite{ACGT08} that $(g_c, \omega_x)$ gives rise to a genuine, non-degenerate, smooth K\"ahler metric on $\bbf_k$.

We notice that $\gz$ is the momentum map with respect to $\omega_x$ of the induced $S^1$-action on $\bbf_k$ corresponding to multiplication on $\calo$ or,  equivalently, the $S^1$-action induced by the push forward of $K=(K_0-K_{\infty})/2$ to the quotient 
space $\bbf_k$. In particular, $\gz$ is a Killing potential.
We have $E_\infty :=P(0\oplus \calo(k)) = \gz^{-1}(-1), E_{0} :=P(\calo \oplus 0) = \gz^{-1}(1),$ and $\bbf_k^0$ corresponds to the $\bbc^*$-bundle over $\bbc \bbp^1$,  obtained from the $\bbc \bbp^1$-bundle $\bbf_k \rightarrow \bbc \bbp^1$ by deleting the zero and infinity sections $E_0$ and $E_\infty$.
 
As observed in \cite{ACGT08}, due to \cite{HwaSi02}, if instead of $\Theta_c(\gz)$ we take in \eqref{g} any smooth function $\Theta(\gz)$ on $[-1,1]$, satisfying
\begin{align}
\label{positivity}
(i)\ \Theta(\gz) > 0, \quad -1 < \gz <1,\quad
(ii)\ \Theta(\pm 1) = 0,\quad
(iii)\ \Theta'(\pm 1) = \mp 2.
\end{align}
then
\begin{equation}\label{metric}
g_x=\frac{1+x\gz}{x}g_{\bbc \bbp^1}+\frac {d\gz^2}
{\Theta (\gz)}+\Theta (\gz)\theta^2,\quad
\omega_x =\frac{1+x\gz}{x}\omega_{\bbc \bbp^1} + d\gz \wedge
\theta,
\end{equation}
with \eqref{theta} yields a smooth $S^1$-invariant K\"ahler metric on $\bbf_k$, compatible with the same symplectic form $\omega_x$. The corresponding  complex structure  is then given on $\bbf_k^0$ by the horizontal lift  of the base complex structure on $\bbc \bbp^1$ (with respect to the chosen Chern connections on $\calo$ and $\calo(k)$) along with the requirement 
\begin{equation}\label{complex}
Jd\gz = \Theta \theta
\end{equation}
on the fibres. Such K\"ahler metrics on $\bbf_k$ are called {\it admissible K\"ahler metrics}.

It is convenient to define a function $F(\mathfrak{z})$ by the formula
\begin{equation}
\Theta(\mathfrak{z})= \frac{F(\mathfrak{z})}{(1+x
\mathfrak{z})}.
\end{equation}
Since $(1+x
\mathfrak{z})$ is positive for $-1<\mathfrak{z}<1$, conditions
\eqref{positivity}
imply the following equivalent conditions on $F(\mathfrak{z})$:
\begin{align}
\label{positivityF}
(i)\ F(\mathfrak{z}) > 0, \quad -1 < \mathfrak{z} <1,\quad
(ii)\ F(\pm 1) = 0,\quad
(iii)\ F'(\pm 1) = \mp 2(1 \pm x).
\end{align}

Let $\mathcal{K}^{\rm adm}(\bbf_k,\omega_x)$ denote the space of all admissible K\"ahler metrics associated to a given choice of $x\in (0,1)$. From the discussion above, $\mathcal{K}^{\rm adm}(\bbf_k,\omega_x)$ is identified with a Fr\'echet space consisting of all smooth functions $\Theta(z)$ on $[-1,1]$ satisfying \eqref{positivity}. The space $\mathcal{K}^{\rm adm}(\bbf_k, \omega_x)$  (associated  to a given choice of $x\in(0,1)$)  can also be equally parametrized by the fibre-wise symplectic potentials $u(\gz)$, where $u(\gz)$ is defined up to an affine-linear  function of $\gz$ by  $u''(\gz)= \frac{1}{\Theta(\gz)}$. It is shown in \cite[p. 559-560]{ACGT08} that the fibre-wise Legendre transform maps  $\mathcal{K}^{\rm adm}(M, \omega_x)$ to the space $\mathcal{K}(\bbf_k, J_c, [\omega_x])=\{\varphi \in C^{\infty}(\bbf_k) : \omega_x + dd^c \varphi >0\}$ of  $J_c$-compatible K\"ahler metrics in the class $[\omega_x]$.
Therefore, with all else fixed, we may view the set of the functions $F$ satisfying \eqref{positivityF} as parametrizing a certain family of K\"ahler metrics within the same K\"ahler class of $\bbf_k$. For each choice of $F$, an appropriate pull-back of $\gz$ remains a Killing potential and hence any affine function of $\gz$ is a Killing potential.

One easily checks that the K\"ahler class of an admissible metric \eqref{metric} is given by
\begin{equation}\label{class1}
[\omega_x] = 4\pi E_{0}^*+ \frac{2\pi(1-x)k}{x} C^*,
\end{equation}
where $C$ denotes a fiber of the ruling $\bbf_k \rightarrow
{\mathbb C}{\mathbb P}^1$ and $D^*$ denotes the Poincar\'e dual of a divisor $D$.
Up to an overall rescale, every K\"ahler class in the K\"ahler cone may be represented by an admissible K\"ahler metric. 

The scalar curvature of an admissible metric $g_x$ as above is given by 
\begin{equation}\label{admscal}
Scal_x=\frac{2sx}{1+x\gz}-\frac{F''(\gz)}{1+x\gz}.
\end{equation}
If $f(\gz)$ is a smooth function of $\gz$ (and hence $f$ can be viewed as a smooth function on $\bbf_k$) and if $\Delta_{g_x}$ denotes the Laplacian of 
$g_x$, then
\begin{equation}\label{admlaplace}
\Delta_x f =-\frac{\frac{d}{d\gz}\left[F(\gz)f'(\gz)\right]}{1+x\gz}.
\end{equation}
Finally, it is also clear that
\begin{equation}\label{normdz}
|d\gz|^2_{g_x}=\frac{F(\mathfrak{z})}{(1+x
\mathfrak{z})}.
\end{equation}

\subsubsection{Answering Question \ref{question2}}\label{answerquestion2}

We will now consider a product metric 
\begin{equation}\label{productmetric}
g=g_{a} + g_x,\quad \omega=\omega_{a}+\omega_x
\end{equation}
on $N_k$, where $a\in \bbq^+$, $g_{a}$ is the (re-scale of the) Fubini-Study metric on $\bbc\bbp^1$ 
with scalar curvature $2a$, and $g_x$ is given by \eqref{metric} such that $x\in (0,1)\cap \bbq$. Naturally $g_{a}$ and $g_x$ are both pulled back to $N_k$.
The corresponding product K\"ahler class on $N_k$, rescaled by $\frac{1}{2\pi}$, is then rational and by varying $a\in \bbq^+$ and $x\in (0,1)\cap \bbq$, while rescaling as appropriate, 
we obtain all possible primitive polarizations of $N_k$. Given such a polarization we consider the Boothby-Wang constructed Sasaki structure on the corresponding $S^1$-bundle over
$N_k$. 

Assume $c\in (-1,1)$ and let $f=c\gz+1$, where $\gz$ (from Section \ref{admreview}) is assumed to be pulled back to $N_k$ and hence a positive Killing potential on $N_k$.
Since the complex dimension of $N_k$ is $3$, we are setting $p=3+2=5$.
Using \eqref{weightedscal}, \eqref{admscal}, \eqref{admlaplace}, \eqref{normdz}, and the fact that we have a product metric, we easily calculate that
$$
Scal_{c\gz+1,5}(g) = (c\gz+1)^2\left(2a+\frac{2sx}{1+x\gz}\right)- (c\gz+1)^2\frac{F''(\gz)}{(1+x\gz)}+8c(c\gz+1)\frac{F'(\gz)}{(1+x\gz)}-20c^2\frac{F(\gz)}{(1+x\gz)}
$$
is a function of $\gz$. Thus $Scal_{c\gz+1,5}(g)$ being a Killing potential, $A_1\gz + A_2$, is equivalent to the following ODE
$$ (c\gz+1)^2F''(\gz) - 8c(c\gz+1)F'(\gz)+20c^2F(\gz)= (c\gz+1)^2\left(2a(1+x\gz)+2sx\right)-(A_1\gz + A_2)(1+x\gz),$$
which simplifies to
\begin{equation}\label{ode}
\frac{d^2}{d\gz^2}\left[\frac{F(\gz)}{(c\gz+1)^4}\right]=(c\gz+1)^{-6}\left( (c\gz+1)^2\left(2a(1+x\gz)+2sx\right)-(A_1\gz + A_2)(1+x\gz)\right)
\end{equation}

\begin{remark}
Note that, generalizing the original definition of \cite{ACGT08}, a metric of $g$'s type
is considered admissible by Definition 7.1 in Section 7.1 of \cite{ApCaLe21}.  As such, \eqref{ode} above
is merely a special case of equation (40) in Section 7.1 of \cite{ApCaLe21}. Further, as is stated in Section 7.1 of \cite{ApCaLe21}, the existence of a solution to \eqref{ode} under the 
endpoint conditions \eqref{positivityF} is indeed a straightforward consequence of the arguments carried out in \cite{ApMaTF18}.

Since we want to carefully explore the existence of constant scalar curvature Sasaki metrics in the Sasaki cone
over a polarization of $N_k$, we are going to elaborate on the solution in this case, confirm that the positivity condition {\em (i)} of \eqref{positivityF} holds (see Lemma \ref{Fispositive} below), and then consider \eqref{cscS}.

\end{remark}

Following Section 2.3 of \cite{ApMaTF18} (with minor adaptations), we can derive that a unique solution to \eqref{ode} under the conditions {\em (ii)} and {\em (iii)} of \eqref{positivityF} (equivalent to \eqref{positivity}) is given by
\begin{equation}\label{thisisF}
F(\gz)= (c\gz+1)^4\left(\frac{2(1-x)}{(1-c)^4}(\gz+1) + \int_{-1}^{\gz}Q(t)(\gz-t)dt\right),
\end{equation}
where 
$$
Q(\gz)=(c\gz+1)^{-6}\left((c\gz+1)^2\left(2a(1+x\gz)+2sx\right)-(A_1\gz + A_2)(1+x\gz)\right),
$$
and $A_1$, $A_2$ are solutions to the system
$$
\begin{array}{ccl}
\alpha_{1,-6} A_1+\alpha_{0,-6} A_2 &=&2\beta_{0,-4}\\
\\
\alpha_{2,-6} A_1+\alpha_{1,-6} A_2 &=&2\beta_{1,-4}
\end{array}
$$
with
$$
\alpha_{r,q}=\int_{-1}^1 (ct+1)^q t^r(1+xt)dt
$$
and
$$
\beta_{r,q}=\int_{-1}^1\left(a(1+xt)+sx\right) t^r(ct+1)^qdt +\left( (-1)^r(1-c)^q(1-x)+(1+c)^q(1+x)\right).
$$
Finally, \eqref{cscS} is then equivalent to $A_1\gz+A_2$ being a constant multiple of $c\gz+1$, which, following Section 3.1 of \cite{BHLT23}, is equivalent to the equation
\begin{equation}\label{cscS2}
\alpha_{1,-5}\beta_{0,-4}-\alpha_{0,-5}\beta_{1,-4}=0.
\end{equation}

Due to Corollary 7.3 of \cite{ApCaLe21} it follows that a ray in the Sasaki cone determined by $c\in (-1,1)$ has extremal Sasaki metrics of any type if and only if the corresponding $F(\gz)$ as defined by \eqref{thisisF} satifies {\em (i)} of \eqref{positivityF}. In that case, the ray has extremal Sasaki metrics which up to scale arise from the construction above.
Furthermore, by Theorem 11.2.8 in \cite{BG05}, we know that if a ray admits a constant scalar curvature Sasaki metric, then the Futaki invariant must vanish and any extremal Sasaki metric for this ray must have constant scalar curvature. As a consequence, \eqref{cscS2} is a necessary
condition for the ray determined by $c\in (-1,1)$ to admit a constant scalar curvature Sasaki metrics of any kind. 

As a slightly degenerate special case of the last statement in Theorem 3.1 of \cite{ApMaTF18} we now have the following lemma:
\begin{lemma}\label{Fispositive}
With $F(\gz)$ given by \eqref{thisisF} the positivity condition {\em (i)} of \eqref{positivityF} holds.
\end{lemma}

\begin{proof}
This follows essentially from the proof of Theorem 3.1 of \cite{ApMaTF18}, which in turn is in depth to the the root counting argument given by Hwang \cite{Hwa94} and Guan \cite{Gua95} (see also Proposition 11 in \cite{ACGT08}). To be more specific, here is the argument:

First notice that the condition {\em (i)} of \eqref{positivityF}  is equivalent to the condition
\begin{equation}\label{positivityG}
 G(\gz)>0,\quad \text{for }-1<\gz <1,
 \end{equation}
 where
$G(\gz)=\frac{F(\gz)}{(c\gz+1)^4}$. It is also easy to see that, {\em (ii)} and {\em (iii)} of \eqref{positivityF} together imply that
\begin{align}
\label{boundaryG}
(i)\ G(\pm 1) = 0,\quad
(ii)\ G'(-1) = \frac{2(1- x)}{(1-c)^4} >0,\quad
(iii)\ G'(+1) = \frac{-2(1+x)}{(1+c)^4}<0.
\end{align}
Assume first that $c\neq 0$. Since $(c\gz+1)>0$ for $-1\leq \gz \leq 1$, equation \eqref{ode} tells us that for $-1\leq \gz \leq 1$, the sign of $G''(\gz)$ equals the sign of $P(\gz)$, where
$P(\gz)$ is the cubic 
$$P(\gz)=(c\gz+1)^2\left(2a(1+x\gz)+2sx\right)-(A_1\gz + A_2)(1+x\gz)$$
with positive leading term, $2axc^2\gz^3$ 
Notice now that $P(-\frac{1}{x})\geq 0$ since $2sx>0$ and so $P(\gz)$ must have a root in the interval $(-\infty,-\frac{1}{x}]\subset (-\infty,-1)$.

Assume by contradiction that \eqref{positivityG} fails. Then \eqref{boundaryG} would imply that $G$  has at least two relative maxima and one minimum;
$-1<\gz_{\text{max}1}<\gz_{\text{min}}<\gz_{\text{max} 2}<1$. This implies that $P(\gz)$ would also have at least one root in each of the intervals $(\gz_{\text{max}1},\gz_{\text{min}})$,
$(\gz_{\text{min}},\gz_{\text{max} 2})$, and $(\gz_{\text{max} 2},+\infty)$. That would give us at least four distinct roots for the cubic $P(\gz)$, which is a contradiction. 
Finally, the $c=0$ case (classical extremal) is even simpler, well-known (\cite{Hwa94}, \cite{Gua95}), and is left for the reader.

In conclusion, \eqref{positivityG} must hold and hence {\em (i)} of \eqref{positivityF}  is proven.
\end{proof}

\begin{lemma}\label{cscexistence}
For any choice of values $k\in\bbz^+$ (yielding $s=2/k$), $a\in \bbq^+$, $x\in (0,1)\cap\bbq$, there is at least one value $c\in (-1,1)$ such that 
\eqref{cscS2} holds true. Furthermore, given values $k\in\bbz^+$ and $x\in (0,1)\cap\bbq$, there exist $N>0$ such that $\forall a \in (N,+\infty)\cap \bbq$,
there are at least three values $c\in (-1,1)$ such that 
\eqref{cscS2} holds true.
\end{lemma}

\begin{proof}
A direct calculation shows that the left hand side of \eqref{cscS2} equals $\frac{4 h(c)}{9(1-c^2)^7 }$, where
$h(c)$ is the following polynomial.
\begin{equation}\label{h}
\begin{aligned}
h(c)  &=  3 x (s x-2)+ \left(21 - 3 a - 3 s x + 3 x^2 + a x^2\right)c\\
&+ 4 x (a-9 - s x)c^2+ 4 (a + s x + 6 x^2 - a x^2)c^3\\
&+  x (sx-6 - 4 a )c^4+\left(3 - a - s x - 3 x^2 + 3 a x^2\right)c^5. 
\end{aligned}
\end{equation}
Since $h(-1)=-24(1+x)^2<0$ and $h(1)=24(1-x)^2>0$, the polynomial clearly has at least one root $c\in (-1,1)$. The last part of the lemma follows from the simple fact that
$$\displaystyle\lim_{a\rightarrow +\infty}\frac{h(-1/2)}{a}=\frac{(33 + 24 x - 3 x^2)}{32} >0\quad \text{and} \quad  
\displaystyle\lim_{a\rightarrow +\infty}\frac{h(1/2)}{a}=\frac{(-33 + 24 x + 3 x^2)}{32}  <0.$$
Thus, for $a\in \bbq^+$ sufficiently large, $h(c)$ has at least one root in each of the intervals $(-1,-1/2)$, $(-1/2,1/2)$, and $(1/2,1)$.
\end{proof}

With Lemmas \ref{Fispositive} and \ref{cscexistence} established we see that for any choice of values $k\in\bbz^+$, $a\in \bbq^+$, $x\in (0,1)\cap\bbq$, and an appropriate rescale of the K\"ahler structure $(g,\omega)$, 
defined by \eqref{productmetric},
such that the K\"ahler class is in $H^2(N_k,\bbz)$, the corresponding Boothby-Wang constructed Sasaki manifold has a sub-cone (parametrized by $c\in (-1,1)$ in the Sasaki cone that is exhausted by extremal Sasaki structures and where at least one (and, for certain polarizations, at least three) of these  extremal structures have constant scalar curvature. Hence the answer to Question \ref{question2} is indeed ``yes.'' We conclude with the following result.

\begin{proposition}\label{affirmative}
For any polarization of $N_k= \bbc\bbp^1 \times \bbf_k$, where $\bbf_k$ is the $k^{th}$ Hirzebruch surface, the corresponding Sasaki manifold over $N_k$ admits a Sasaki structure with constant scalar curvature somewhere in the Sasaki cone. Furthermore, for each $N_k$, there exist polarizations such that the corresponding Sasaki cone contains at least three constant scalar curvature Sasaki rays. 
\end{proposition}

\begin{remark}
The findings of multiple constant scalar curvature Sasaki metrics is a result similar to that in Theorem 1.3 in \cite{BoTo14a}. Note also that such a phenomenon was first discovered by Legendre \cite{Leg10} in the case of certain $5$-dimensional Sasaki manifolds corresponding to quadrilateral toric structures.
\end{remark}

\subsection{Related cases}
\subsubsection{Dimension 7}
We can generalize part of the work in Section \ref{answerquestion2} by replacing $\bbc\bbp^1$ in $N_k=\bbc\bbp^1\times \bbf_k$ by a compact Riemann surface, $\Sigma_{\gg_1}$, of arbitrary genus $\gg_1$and $\bbf_k$ by any admissible ruled surface, $\bbs_k=\bbp(\calo \oplus L_k)\rightarrow \Sigma_{\gg_2}$, where $\Sigma_{\gg_2}$ is any compact Riemann surface of genus $\gg_2 $ and $L_k$ is a holomorphic line bundle of degree $k$.
This corresponds to allowing $a$ to take any rational value and setting $s=\frac{2(1-\gg_2)}{k}$ in Section \ref{admreview}. 
The first part of Lemma \ref{cscexistence} holds true, regardless of the values of $s$ and $a$, whereas Lemma \ref{Fispositive} is a bit more complicated. 

\begin{example}
Let us first keep $\bbf_k$ as is, let $\Sigma_{\gg_1}$ be a 2 dimensional torus $T^2$, and let $g_a$ be a scalar flat K\"ahler metric on $T^2$. Then $s=2/k>0$  and $a=0$. In this case,
Lemma \ref{Fispositive} still holds true. The proof of this is almost the same as above, but now $P(\gz)$ is a second order polynomial. We still have that $P(\frac{-1}{x}) \geq 0$ and if we assume by contradiction that \eqref{positivityG} fails we still have that \eqref{boundaryG} would imply that $G$  has at least two relative maxima and one minimum;
$-1<\gz_{\text{max}1}<\gz_{\text{min}}<\gz_{\text{max} 2}<1$. This implies that $P(\gz)$ would also have at least one root in each of the intervals $[\frac{-1}{x}, \gz_{\text{max}1})$,
$(\gz_{\text{max}1},\gz_{\text{min}})$,
$(\gz_{\text{min}},\gz_{\text{max} 2})$. This is obviously impossible for a second order polynomial. We thus have the following result.
\end{example}

\begin{proposition}\label{affirmative2}
For any polarization of $T^2 \times \bbf_k$, where $\bbf_k$ is the $k^{th}$ Hirzebruch surface, the corresponding Sasaki manifold over $T^2 \times \bbf_k$ admits a Sasaki structure with constant scalar curvature somewhere in its Sasaki cone. \end{proposition}

\begin{remark}\label{cPositive}
Before embarking on the next examples, note that in the proof of Lemma \ref{cscexistence}, we easily see that $h(0)=3x(sx-2)$ and since $sx<2$ for any value of $\gg_2 \in \bbz^{>0}$, this means that $h(0)<0$. Thus, since $h(1)>0$, we know that $h(c)$ has at least one positive root $c\in (0,1)$.
\end{remark}

\begin{example}\label{cpos}
Assume now that $\gg_2 \in \bbz^{>0}$ is arbitrary,  $a>0$ (so, $\Sigma_{\gg_1}=\bbc\bbp^1$), and $c\in (0,1)$ is a positive solution of $h(c)=0$. By Remark \ref{cPositive}, we know such a solution exist. We now break our discussion into 2 cases $c\neq x$ and $c=x$.

In the first case we assume that $c\neq x$. Then $h(c)=0$ can be rewritten as the equation
$$s=\frac{3 \left(c^5 x^2-c^5+2 c^4 x-8 c^3 x^2+12 c^2 x-c x^2-7 c+2 x\right) +c(1-c^2)(3 - c^2 - 4 c x - x^2 + 3 c^2 x^2)a}{\left(1-c^2\right)  \left(3-c^2\right) x (x-c)}$$ 
and then
$F(\gz)$ given by \eqref{thisisF} may be written as follows:
\begin{equation}\label{explicitF}
F(\gz)=\frac{(1-\gz^2)(1+c\gz) g(x,\gz)}
{2(1-c^2)(3-c^2)},
\end{equation}
where
$$
\begin{array}{ccl}
g(x,\gz)& = & 2 (1 - x) \left(3 + c^2 - 3 c \gz + c^3 \gz - 2 c^2 \gz^2\right)\\
\\
&+ & x(1 - c) \left(1 + 
   \gz\right) \left((3(1 - c) + a c(1+c) ) (1 - \gz) + (3 - c^2) (1 + \gz)\right)
   \end{array} $$
 is linear in $x$. 
 It is now a straightforward exercise to check that $g(0,\gz)$ and $g(1,\gz)$ are both positive for $-1<\gz <1$ and thus
 $g(x,\gz)$ is positive for $0<x<1$ and $-1<\gz <1$. This proves that $F(\gz)$ satisfies {\em (i)} of \eqref{positivityF}.
 
For the second case we assume that $c=x$ is a solution of $h(c)=0$. It is easy to see that then we must have $a=\frac{5-x^2}{1-x^2}$.
In that case, \eqref{thisisF} simplifies to
$F(\gz)=\frac{(1-z^2)(1-x z)(1+x z)^2}{1-x^2}$
which clearly satisfies {\em (i)} of \eqref{positivityF}. 

We can now conclude with the following existence result:
\begin{proposition}\label{affirmative3}
Let $N = \bbp(\calo \oplus L)\rightarrow \Sigma_{\gg}$ be a ruled surface over a compact Riemann surface, $\Sigma_{\gg}$.
For any value of $\gg \in \bbz^{>0}$ and any choice of polarization of $\bbc\bbp^1\times N$, the corresponding Sasaki manifold over $\bbc\bbp^1\times N$ admits a Sasaki structure with constant scalar curvature somewhere in the Sasaki cone. 
\end{proposition}

 \end{example}
 
\begin{example}\label{quasiregularex}
Let us focus a bit more on the second case in Example \ref{cpos}.
Recall from above that $c=x$ is a solution of $h(c)=0$ if and only if $a=\frac{5-x^2}{1-x^2}$. In particular, it follows automatically that $a>0$ and so $\Sigma_{\gg_1}$ must be $\bbc\bbp^1$.
We saw above that in this case \eqref{thisisF} simplifies to
$F(\gz)=\frac{(1-z^2)(1-x z)(1+x z)^2}{1-x^2}$
and since this satisfies {\em (i)} of \eqref{positivityF}, we get a family of polarizations of $\bbc\bbp^1\times \bbs_k$ such that the corresponding Sasaki manifold has an explicit quasi-regular cscS ray.

We can explore this family a little further in two directions: 
\begin{enumerate}
\item If $\gg_2$ and $k$ are such that $5k<3(\gg_2-1)$ (in particular $\gg_2\geq 3$) and $x=\frac{k}{\gg_2-1}$, then $0<x<\frac{3}{5}$ and thus
$\frac{-5x}{3} \in (-1,1)$. Since $x=\frac{k}{\gg_2-1}$ we have that $s=\frac{2(1-\gg_2)}{k}=\frac{-2}{x}$ and then with $c=\frac{-5x}{3}$ in \eqref{thisisF}
we get
$F(\gz)=\frac{(1-z^2)(1-x z)(1+x z)^2}{1-x^2}$, similarly to when $c=x$. This means that 
for this particular polarization determined (up to scale) by  $x=\frac{k}{\gg_2-1}$ and $a=\frac{5-x^2}{1-x^2}$, the K\"ahler metric given by \eqref{productmetric} (up to scale) is
both $(f_1,5)$-extremal and $(f_2,5)$-extremal, where $f_1=x\gz+1$ and $f_2=\frac{-5x}{3} \gz+1$. In other words, $(g,f_1)$ and $(g,f_2)$ are
$5$-weighted extremal K\"ahler twins as defined in \cite{BHLT25} and recalled in Definition \ref{twinsKahler} below. Note that $h(c)$ from \eqref{h} is not zero when $c=\frac{-5x}{3}$.
On the Sasaki level, we can say that the two rays determined by $f_1$ and $f_2$ in the Sasaki cone of the Sasaki manifold, corresponding to this specific polarization, are
extremal Sasaki twins as defined in \cite{BHLT25} and recalled in Definition \ref{twinsSasaki} below, with one of the extremal Sasaki metrics being cscS (whereas the other is not).
We summarize this in the following proposition:
\end{enumerate}

\begin{proposition}
Let $N=\bbp({\mathcal O} \oplus L) \rightarrow \Sigma_{\gg}$ be a ruled surface over a compact Riemann surface, $\Sigma_{\gg}$, with genus $\gg \geq 3$.
Assume that the holomorphic line bundle $L \rightarrow \Sigma$ satisfies that $\int_\Sigma c_1(L) =k$ such that $k<\frac{3(\gg -1)}{5}$.
Then $\bbc\bbp^1 \times N$ has an integer K\"ahler class $\Omega$ with a K\"ahler metric representative $(g,\omega)$ and two
distinct (even up to scale) positive, smooth Killing potentials, $f_1, f_2$ such that $g$ is both $(f_1,5)$- and $(f_2,5)$-extremal.
That is, $(g,f_1)$ and $(g,f_2)$ are $5$-weighted extremal twins.
Moreover, on the ($7$-dimensional) Boothby-Wang related Sasaki manifold,  the corresponding extremal Sasaki twins are both quasi-regular. One of them has
constant scalar curvature whereas the other does not.
\end{proposition}

\begin{enumerate}[resume]

\item If $\gg_2=11$, $k=9$ (so $s=\frac{-20}{9}$), $x=9/10$ and $a=\frac{5-x^2}{1-x^2}=\frac{419}{19}$ we see that
see that $h(c)$ as defined in \eqref{h} is given by
$$h(c)=\frac{3}{475} (10 c-9) \left(540 c^4-885 c^3-350 c^2+543 c+190\right)$$
and we see that $h(c)$ has three solutions in $(-1,1)$; $c_1=9/10$, $c_2\approx -0.359$, and $c_3\approx -0.601$.

Further, we see that with $s=\frac{-20}{9}$, $x=9/10$, and $a=\frac{419}{19}$, \eqref{thisisF} implies that
$$F(\gz)= \frac{(1-\gz^2) g(\gz)}{190  \left(365 - 720 c +394 c^2-35 c^4\right)},$$
where
$$
\begin{array}{ccl}
g(c,\gz) & = & 10 (1805 - 3363 c + 22143 c^2 - 30840 c^3 + 10800 c^4)\\
\\
&+& 5 (12483 - 28728 c + 50505 c^2 - 45860 c^3 + 12600 c^4)\gz\\
\\
&+& 10 (5130 - 10317 c - 14657 c^2 + 30840 c^3 - 11465 c^4)\gz^2\\
\\
&+&  (20520 - 185151 c + 229300 c^2 - 68985 c^3)\gz^3
\end{array}.
$$
\end{enumerate}
Thus, for a given $c\in (-1,1)$, $F(\gz)$ satisfies {\em (i)} of \eqref{positivityF} precisely when $g(c,\gz)$ is positive for all $-1<\gz<1$. 
This is certainly true for  $g(c_1,\gz)=26353 (10 - 9 z) (10 + 9 z)^2/2000$ and numerically it is also evident that positivity holds\footnote{It is relatively easy to check for that for $-1<c<1$, $g(c,\pm 1)>0$. Since $g(c,\gz)$ is a cubic in $\gz$, it must have a real root $|\gz|>1$.
Further, the discriminant of $g(c,\gz)$ (as a polynomial in $\gz)$) is negative for both $c_2$ and $c_3$ and thus $g(c_i,\gz)$ (which now has precisely one real root) is positive for $-1<\gz<1$ and $i=2,3.$}
for $g(c_2,\gz)$ and $g(c_3,\gz)$. Thus we have three cscS rays in a $2$-dimensional subcone of the corresponding Sasaki cone. On the other hand
$g(0,\gz)=95 (190 + 657 \gz + 540 \gz^2)$ and this is negative for e.g. $\gz=-3/5$. So the regular ray ($c=0$) is not Sasaki extremal.
This means that the intersection of the extremal Sasaki cone and the $2$-dimensional subcone parametrized up to scale by
$-1<c<1$ is disconnected and the cscS rays corresponding to $c_2$ and $c_3$ are not in the same component as the cscS ray corresponding to $c_1$. Of course, here the entire Sasaki cone is $3$-dimensional, so this
might not mean that the three rays are not in the same connected component of the entire extremal Sasaki cone. In Example \ref{cscmoatseparated} below we shall see a case of a Sasaki manifold (of dimension $9$) with a $2$-dimensional Sasaki cone
that have two cscS-rays that are not in the same connected component of the extremal Sasaki cone.
\end{example}

\begin{example}\label{nocscex}
Now let us chose $\Sigma_{\gg_1}$ to have genus at least two and let us scale $g_a$ such that $a=-\frac{43137}{1337}$. For $\Sigma_{\gg_2}$  and $L_k$ we will chose $\gg_2>1$  and $k$ such that 
$s=-200$. For instance we could let $\gg_2=101$ and $k=1$. We call the resulting K\"ahler manifold $\tilde{N}$.
If we pick $x=1/2$, it is straightforward to see that $h(c)$ as defined in \eqref{h} has precisely one root in
$(-1,1)$, namely $c=2/5$. Thus this is is the only value of $c\in (-1,1)$ where \eqref{cscS2} and hence \eqref{cscS} holds.
With this data,
$$F(\gz)=\frac{(1-\gz^2) (2 \gz+5) \left(1820 \gz^2+191 \gz-292\right)}{8022}$$
and it is straightforward to see that $F(\gz)$ fails {\em (i)} of \eqref{positivityF} miserably.
This implies that the corresponding $\Theta(\gz)=\frac{F(\gz)}{(1+x\gz)}$ is not a positive function for $-1<\gz<1$.
Due to the fact that the genera of both $\Sigma_{\gg_1}$ and $\Sigma_{\gg_2}$ are at least two we know that the Sasaki cone for the Boothby-Wang constructed Sasaki manifold over $\tilde{N}$ with respect 
to an appropriate rescale of $g=g_a+g_x=g_{(-\frac{43137}{1337})}+g_{\frac{1}{2}}$ has dimension two. As such the Sasaki rays are parametrized by $c\in (-1,1)$. 
In fact, it can be checked that $F(\gz)$ given by \eqref{thisisF} is not satisfying  {\em (i)} of \eqref{positivityF} for any value $c\in (-1,1)$ and hence, due to Theorem 3/Corollary 7.3 in \cite{ApCaLe21}, the entire Sasaki cone is void of extremal
Sasaki metrics. 

We conclude that the Sasaki cone has no constant scalar curvature 
Sasaki metric at all and by Remark \ref{cscSinMi}, we have arrived at an example where the Sasaki cone of a join $M_1\star_{l_1,l_2} M_2$ does not have a constant scalar curvature ray despite the fact that
$M_1$ and $M_2$ each admit a cscS ray somewhere in their individual Sasaki cones. 
 \end{example}

\begin{example}\label{endonposnote} To end on a positive note, let us chose $\Sigma_{\gg_1}$ to have genus at least two and let us scale $g_a$ such that $a=-\frac{2675}{497}$. For $\Sigma_{\gg_2}$  and $L_k$ we will chose $\gg_2>1$  and set $k=\gg_2-1$ such that 
$s=-2$. 
If we pick $x=1/2$, it is straightforward to see that $h(c)$ as defined in \eqref{h} has precisely one root in
$(-1,1)$, namely $c=1/8$. Thus this is the only value of $c\in (-1,1)$ where \eqref{cscS2} and hence \eqref{cscS} holds.
With this data,
$$F(\gz)=\frac{(1-\gz^2) (\gz+8) \left(29 \gz^2+142 \gz+326\right)}{2982}$$
and it is straightforward to see that $F(\gz)$ satisfies {\em (i)} of \eqref{positivityF}.
Thus in this case the Sasaki manifold over $\Sigma_{\gg_1}\times \bbs_{\gg_2-1}$ corresponding to the polarization determined up to scale by $x$ and $a$ above, admits a Sasaki structure with constant scalar curvature somewhere in the Sasaki cone. 
\end{example}

\subsubsection{Higher Dimension Related Cases}\label{higherdim}
We can generalize the preliminary work in Section \ref{answerquestion2} even further by replacing $\bbc\bbp^1$ in $N_k=\bbc\bbp^1\times \bbf_k$ by any compact K\"ahler manifold, $N_1$, of complex dimension $d$ and  equipped with a CSC K\"ahler metric $g_a$ with constant scalar curvature\footnote{It would perhaps be more natural to write the scalar curvature as $2da$, but it is more convenient to stick with $2a$ here.} $2a$. We assume that after an appropriate rescale, the K\"ahler class of the K\"ahler form of $g_a$ is
an integer class. We will continue to replace 
 $\bbf_k$ by any admissible ruled surface, $\bbs_k=\bbp(\calo \oplus L_k)\rightarrow \Sigma_{\gg}$, where $\Sigma_{\gg}$ is any compact Riemann surface of genus $\gg$ and $L_k$ is a holomorphic line bundle of degree $k$.
We still have that $s=\frac{2(1-\gg)}{k}$ in Section \ref{admreview}. We just need to adjust the equations and subsequent $F(\gz)$ a bit to the fact that
now $p$ in \eqref{weightedscal} should more generally be $(d+2)+2=d+4$.

It is not hard to see that the adaptation of the work in \ref{answerquestion2} leads to
\begin{equation}\label{thisisFmoregenerally}
F(\gz)= (c\gz+1)^{p-1}\left(\frac{2(1-x)}{(1-c)^{p-1}}(\gz+1) + \int_{-1}^{\gz}Q(t)(\gz-t)dt\right),
\end{equation}
where 
$$
Q(\gz)=(c\gz+1)^{-(p+1)}\left((c\gz+1)^2\left(2a(1+x\gz)+2sx\right)-(A_1\gz + A_2)(1+x\gz)\right),
$$
and $A_1$, $A_2$ are solutions to the system
\begin{equation}\label{A1A2}
\begin{array}{ccl}
\alpha_{1,-(p+1)} A_1+\alpha_{0,-(p+1)} A_2 &=&2\beta_{0,-(p-1)}\\
\\
\alpha_{2,-(p+1)} A_1+\alpha_{1,-(p+1)} A_2 &=&2\beta_{1,-(p-1)},
\end{array}
\end{equation}
with
$
\alpha_{r,q}$
and
$
\beta_{r,q}
$ 
given exactly as before.
Finally, \eqref{cscS2} is adjusted to the more general cscS equation
\begin{equation}\label{cscS2moregenerally}
\alpha_{1,-p}\beta_{0,-(p-1)}-\alpha_{0,-p}\beta_{1,-(p-1)}=0.
\end{equation}

\begin{example}\label{cscmoatseparated}
Let us now assume that $N_1 = \bbc\bbp^2\#q\overline{\bbc\bbp}^2$
for $4 \leq q \leq 8$ ($\bbc\bbp^2$ blown up at
$q$ generic points).  Then $N_1$ has complex dimension $d=2$ and admits a KE metric $g_a$
which, for any choice of $a\in \bbq^+$, can be rescaled such that the scalar curvature equals $2a$. 
Given such an $a$ and a choice of $x\in (0,1)\cap \bbq$, we have (after an appropriate rescaling of the product K\"ahler form $\omega_a+\omega_x$) a polarization of 
$N_1\times \bbs_k$ and, assuming $\gg>1$, the $2$-dimensional Sasaki cone of the resulting $9$-dimensional Sasaki manifold is 
parametrized, up to scale, by $-1<c<1$.

We now choose $k$ and $\gg$ such that $s=-3$ (e.g. we could have $k=2$ and $\gg=4$) and let $a$ and $x$ be chosen such that
$a=\frac{3 \left(x^4+7\right)}{\left(1-x^2\right) \left(3-x^2\right)}$. In that case, the left hand side of \eqref{cscS2moregenerally} equals
$$\frac{4 (x-c) h(x,c)}{15 (1-c^2)^9  (1-x^2)  \left(3-x^2\right)},$$
where 
$$
\begin{array}{ccl}
h(x,c)& = & (-12 + 9 x + 8 x^2 - 12 x^3 + 20 x^4 + 3 x^5) c^6\\
&+ &2 x (9 - 16 x^2 + 15 x^4)c^5\\
&+ & (-18 - 27 x - 72 x^2 + 36 x^3 - 38 x^4 - 9 x^5)c^4\\
&+& 8 x (5 - x^2) (9 - 7 x^2)c^3\\
&+ &(-228 + 63 x + 408 x^2 - 84 x^3 - 68 x^4 + 21 x^5)c^2\\
&- & 10 x (9-x^2) c\\
& - &5 (1 -x^2 ) (2 + 3 x) (3 - x^2).
\end{array}
$$

Thus \eqref{cscS2moregenerally} is always solved by $c=x$ and any solutions beyond this, would be roots $-1<c<1$ of $h(x,c)$ as a polynomial in $c$.
First we note that with $c=c_1:=x$, $a=\frac{3 \left(x^4+7\right)}{\left(1-x^2\right) \left(3-x^2\right)}$, and $s=-3$, equation \eqref{thisisFmoregenerally}
gives us that 
$$F(\gz) =\frac{ (1 - \gz^2) (1 + x \gz)^2 (3 + x^2 - x(3-x^2) \gz  - 
    2 x^2 \gz^2)}{(1 - x^2)(3 - x^2)},$$
 which is easily seen to satisfy  {\em (i)} of \eqref{positivityF}. Thus we have at least one cscS ray in the Sasaki cone.
As we vary the value of $0<x<1$ we notice that the nature of $h(x,c)$ will change. If $x\in (0,1)\cap \bbq$ is sufficiently small, $h(x,c)<0$ for $-1<c<1$ and thus no further cscS rays exist in the 
corresponding Sasaki cone. On the other hand e.g. 
$$
\begin{array}{ccl}
h(8/10,c)&= &\frac{2 \left(5236 c^6+12260 c^5-134003 c^4+197072 c^3+30532 c^2-107380 c-29205\right)}{3125}\\
\\
\text{and}\\
\\
h(9/10,c)&=& \frac{3 \left(290849 c^6+352890 c^5-3487407 c^4+3348648 c^3+2190983 c^2-2503170 c-325945\right)}{100000}
\end{array}
$$
each have two roots in $(-1,0)$. This can for example be seen directly by noticing that for both $x=8/10$ and $x=9/10$ we have $h(x,0)<0$, $h(x,-2/5)>0$, and $h(x,-9/10)<0$.

Now, for $x=8/10$, $a=\frac{3 \left(x^4+7\right)}{\left(1-x^2\right) \left(3-x^2\right)}=4631/177$, and $s=-3$, $F(\gz)$ given by \eqref{thisisFmoregenerally}, 
is numerically seen to satisfy  {\em (i)} of \eqref{positivityF} for all values of $c\in (-1,1)$. This means that the Sasaki cone is exhausted by extremal Sasaki rays and we
have three distinct cscS rays in this cone.

However, for $x=9/10$,  $a=\frac{3 \left(x^4+7\right)}{\left(1-x^2\right) \left(3-x^2\right)}= 76561/1387$, and $s=-3$, $F(\gz)$ given by \eqref{thisisFmoregenerally}, 
does not always satisfy  {\em (i)} of \eqref{positivityF}. Numerically it can be seen that there exists a small closed interval $I=(c_l,c_r) \subset (-1,1)$ with $c_l<0<c_r$
such that for
$c\in I$, $F(\gz)$ does not satisfy {\em (i)} of \eqref{positivityF} while for $c\in (-1,1)\setminus I$, $F(\gz)$ satisfies {\em (i)} of \eqref{positivityF}. 
Thus, due to Theorem 3/Corollary 7.3 in \cite{ApCaLe21}, the extremal Sasaki cone is disconnected
and is represented up to scale by two open intervals $I_l=(-1,c_l)$ and $I_r=(c_r,1)$.
Any ray in the ``moat'' determined by $c\in I$ has no extremal Sasaki metric at all. Finally, numerically the roots of $h(9/10,c)$ in $(-1,0)$ are
$c_2\approx -0.120$ and $c_3\approx -0.786$. It can be checked that $c_1=9/10\in I_r$, $c_2\in I$, and $c_3\in I_l$. Thus $c_1$ and $c_3$ correspond to genuine cscS rays, whereas $c_2$ does not.
The two cscS rays are thus separated by a non-extremal moat in the Sasaki cone.
Figure \ref{moat} is a sketch of the situation where the two dimensional Sasaki cone is viewed as the open first quadrant in the $(w_1,w_2)$-plane and the 
ray corresponding to a value $c\in (-1,1)$ is given by line through the origin with slope equal to $\frac{1-c}{1+c}$. That is $c=\frac{w_1-w_2}{w_1+w_2}$.

\bigskip

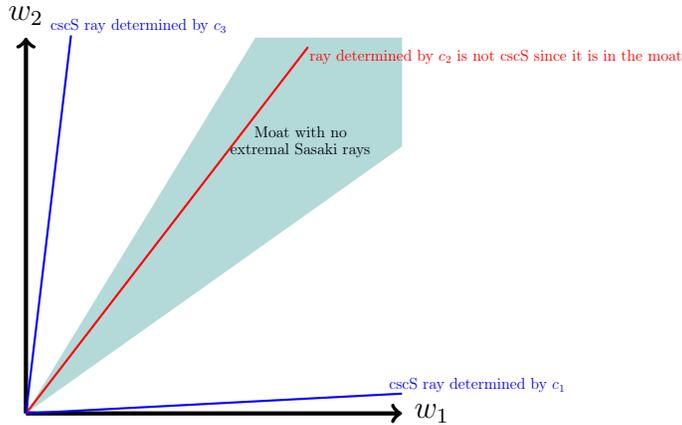
\begin{figure}
\begin{tikzpicture}[scale=5]
 \draw[->,ultra thick] (0,0)--(1,0) node[right]{$w_1$};
\draw[->,ultra thick] (0,0)--(0,1) node[above]{$w_2$};
     \fill[teal, opacity=0.3] (0,0) -- (1,0.71) --  (1,1) -- (0.61,1) -- cycle;
       \draw[scale=1,domain=0:1,smooth,variable=\x,blue, thick] plot ({\x},{(\x)/19}) ;
          \node[scale=0.5, blue] at (1.2,0.075) {cscS ray determined by $c_1$};  
         \draw[scale=1,domain=0:0.75,smooth,variable=\x,red, thick] plot ({\x},{1.3*(\x)}) ;
        \node[scale=0.5, red] at (1.25,0.95) {ray determined by $c_2$ is not cscS since it is in the moat};
        \draw[scale=1,domain=0:0.12,smooth,variable=\x,blue, thick] plot ({\x},{8.4*(\x)}) ;
        \node[scale=0.5, blue] at (0.3,1.03) {cscS ray determined by $c_3$};
         \node[scale=0.5, black] at (0.73,0.75) {Moat with no};
 \node[scale=0.5, black] at (0.73,0.7) {extremal Sasaki rays};
   \end{tikzpicture} 
 \caption{Example \ref{cscmoatseparated} with $x=9/10$,  $a= 76561/1387$}  \label{moat}
   \end{figure}
 
 \end{example}
 
 \begin{example} \label{resurrection}
 Recall the example presented in Example \ref{nocscex}. This case demonstrates the existence of a join-constructed Sasaki manifold whose resulting Sasaki cone admits no extremal Sasaki metrics at all, much less a cscS ray.
 Suppose we now take this Sasaki manifold as $M_2$ in a regular join $M_1\star_{l_1,l_2} M_2$, where $M_1$ is $S^3$. This is equivalent to choosing $N_1=\bbc\bbp^1 \times \Sigma_{\gg_1}$ and $N_2=\bbs_k \rightarrow \Sigma_{\gg_2}$,
 where $\Sigma_{\gg_1}$, $\Sigma_{\gg_2}$, $k$ (hence $s=-200$), and $x=1/2$ are as in Example \ref{nocscex} and the K\"ahler metric chosen on $N_1=\bbc\bbp^1 \times \Sigma_{\gg_1}$  is $g_1=g_{q}+g_{a}$, where
 $g_a$ is the K\"ahler metric on $\Sigma_{\gg_1}$ chosen in Example \ref{nocscex}, with scalar curvature $2a=-2\cdot \frac{43137}{1337}$, and $g_{q}$ is a rescale of the Fubini-Study metric on $\bbc\bbp^1$ with scalar curvature $2q\in \bbq^+$.
 Note that the value of $q$ in connection with the already chosen values of $a$ and $x$ as well as the genera of $\Sigma_{\gg_i}$, $i=1,2$ will determine a polarization of $N_1\times N_2$ and hence the values of co-prime $(l_1,l_2)$ in the join. 
 Since there are several workable choices of the genera in Example \ref{nocscex}, we will leave this aspect of the current example up to the reader.
 
 Moving on, we now see that this situation can be handled by the set-up in the second paragraph of Section \ref{higherdim}: We simply need to set $p=6$, $s=-200$, $x=1/2$, and (the new) $a=q-\frac{43137}{1337}$ in equations
 \eqref{thisisFmoregenerally},
 \eqref{A1A2}, and \eqref{cscS2moregenerally} as well as in the associated expressions for $Q(\gz)$, $\alpha_{r,q}$, and $\beta_{r,q}$.  We can now, for example, calculate that with $p=6$, $s=-200$, $x=1/2$, and $a=q-\frac{43137}{1337}$, 
 the cscS equation \eqref{cscS2moregenerally} is solved by $c=3/5$ when we let $q=125919069/1574986$. For this data, \eqref{thisisFmoregenerally} and \eqref{A1A2} implies that
 $$F(\gz)=\frac{(1-\gz^2)(3 \gz+5) \left(413335+59909 \gz-297891 \gz^2-76401 \gz^3\right)}{527744},$$
 which is clearly positive for $-1<\gz <1$.  Thus, we obtain an cscS metric in the Sasaki cone of the new join. In other words, while we have non-existence of cscS metrics for the join in Example \ref{nocscex}, we can ``resurrect'' cscS existence by joining this join by $S^3$ in a certain way. 
Note that $F(\gz)$ given by \eqref{thisisFmoregenerally} and \eqref{A1A2} and
$p=6$, $s=-200$, $x=1/2$, and $a=q-\frac{43137}{1337}$, with $q=125919069/1574986$,
fails the positivity condition in \eqref{positivityF} for many $c$ values in $(-1,1)$. In other words, the Sasaki cone of the new join is not exhausted by extremal Sasaki metrics.
 \end{example}

\section{Twins}\label{twins}

In \cite{BHLT25} the notion of {\it weighted extremal K\"ahler twins} was introduced as follows.   

\begin{definition}\label{twinsKahler}
Let $(N,J,g,\omega)$ be a K\"ahler manifold of complex dimension $n$ and let $p\in \bbr$. Suppose that $f_1,f_2$ are two positive, smooth Killing potentials on $(N,J,g,\omega)$ that are not constant rescales of each other. 
If $g$ is simultaneously $(f_1,p)$-extremal and $(f_2,p)$-extremal then we say that $(g,f_1)$ and $(g,f_2)$ are $p$-\emph{weighted extremal K\"ahler twins}. 
\end{definition}

We recall the natural companion definition to Definition \ref{twinsKahler}, when $p=n+2$, from \cite{BHLT25}:
\begin{definition}\label{twinsSasaki}
Let $(M,\cald, J)$ be a pseudo-convex CR structure of Sasaki type on a compact smooth manifold $M$ of dimension $2n+1$. \emph{Extremal Sasaki twins} of $(\cald,J)$ are two extremal Sasaki structures both compatible with $(\cald,J)$ and having commuting non-colinear Sasaki-Reeb vector fields. 
\end{definition}

It is well-known that for any positive Killing potential $f$ of $(\bbc\bbp^n, g,J,\omega)$ with its Fubini-Study K\"ahler metric $g$, $(g,f)$ and $(g,1)$ are $(n+2)$-weighted extremal twins.
In \cite{BHLT25} we gave a direct proof using
the toric formalism and a simple calculation, but the result originates from \cite{Bry01,ApCaGa06}  and is due to the fact that the CR-structure of the $(2n+1)$-dimensional round sphere above $\bbc\bbp^n$
governs the Bochner-flatness of the (local) K\"ahler quotients which in turn ensures extremality. In Section 3.1 in \cite{ApCa18} this is explained and in particular we have from Proposition 1 of \cite{ApCa18} that
if $(N,J,g,\omega)$ is a Bochner flat K\"ahler manifold then it is an $(f,n+2)$-extremal K\"ahler structure for any positive smooth Killing potential $f$ on $(N,J,g,\omega)$.
Thus, we also have that if $(N,J,g,\omega)= N^d_{-c}\times N^{n}_c$, where $N^d_{-c}$ is a compact K\"ahler manifold of complex dimension $d$ and constant holomorphic sectional curvature $-c$
and $N^{n}_c$ is a compact K\"ahler manifold of complex dimension $n$ and constant holomorphic sectional curvature $c$, then  $(N,J,g,\omega)$ is an $(f,d+n+2)$-extremal K\"ahler structure for any positive smooth Killing potential $f$ on $(N,J,g,\omega)$.
This follows from Section 2 of \cite{TaLi70} which shows that $(N,J,g,\omega)= N^d_{-c}\times N^{n}_c$ is Bochner flat.

Inspired by this, we will show the following statement.

\begin{proposition}\label{manytwins}
Let $(N_1, g_1, \omega_1)$ be a compact K\"ahler manifold of complex dimension $d$ with constant scalar curvature $Scal_1 =\frac{-2d(d+1)}{n+1}$ and
let $(N_2, g_2, \omega_2)=(\bbc\bbp^n, g_{FS}, \omega_{FS})$ denote the complex projective space of complex dimension $n$ with the Fubini-Study metric of constant scalar curvature 
$Scal_{FS}=2n$.
Let $f$ be the pullback to $N_1\times \bbc\bbp^n$ of any positive smooth $\bbt^n$-invariant Killing potential $f$ on $(\bbc\bbp^n, g_{FS}, \omega_{FS})$. 
Then $(N_1\times \bbc\bbp^n, g_1+g_{FS},\omega_1+\omega_{FS})$
is an $(f,d+n+2)$-extremal K\"ahler structure.
\end{proposition}

\begin{remark}
Note that holomorphic sectional curvature of  $(\bbc\bbp^n, g_{FS}, \omega_{FS})$ is given by 
$$\frac{Scal_{FS}}{2n(n+1)}=\frac{2n}{2n(n+1)}=\frac{1}{n+1}.$$
If $(N_1, g_1, \omega_1)$ also had
constant holomorphic sectional curvature, this holomorphic sectional curvature would be equal to 
$$\frac{Scal_1}{2d(d+1)}=\frac{\frac{-2d(d+1)}{n+1}}{2d(d+1)}=\frac{-1}{n+1}$$ 
and thus, in this special case, the proposition would follow from the discussion above.
However, we do not need to assume that $(N_1, g_1, \omega_1)$ has constant holomorphic sectional curvature in order to prove the proposition above.
\end{remark}

\begin{proof}[Proof of Proposition \ref{manytwins}]
We only need to make very small adjustments to the calculations in Section 5.4 of \cite{BHLT25}, but for the convenience of the reader we will write up the details.

We use the standard toric set-up for $\bbc\bbp^n$. The Delzant polytope for $\bbc\bbp^n$ is a $n$-simplex defined by $\ell_1(x) = 1+x_1$, $\dots$, $\ell_n(x) = 1+x_n$ and $\ell_0(x) = 1-x_1-\dots-x_n$, the symplectic potential is known, from \cite{Abr01}, to be $$\varphi(x) = \frac{1}{2}\sum_{k=0}^n \ell_k(x) \log \ell_k(x)$$ and, thus, its Hessian is 
$${\bf G} = \frac{1}{2\ell_0} \begin{pmatrix}
    \frac{\ell_0}{\ell_1} +1 & 1 &1 &\dots &1\\
    1&\frac{\ell_0}{\ell_1} +1 & 1& \dots &1\\
    1&1 & \dots & \dots &1\\
    && \vdots&&\\
     1&&&1& \frac{\ell_0}{\ell_n} +1 \\
\end{pmatrix}$$ with inverse 

$${\bf H} = \begin{pmatrix}
    2\ell_1 -\frac{2\ell_1^2}{n+1} & -\frac{2\ell_1\ell_2}{n+1} &  -\frac{2\ell_1\ell_3}{n+1} &\dots & -\frac{2\ell_1\ell_n}{n+1} \\
     -\frac{2\ell_1\ell_2}{n+1} &2\ell_2 -\frac{2\ell_2^2}{n+1} &  -\frac{2\ell_3\ell_2}{n+1}& \dots & -\frac{2\ell_n\ell_2}{n+1}\\
    \vdots&& \vdots&&\vdots\\
    - \frac{2\ell_1\ell_n}{n+1} &\dots && -\frac{2\ell_{n-1}\ell_n}{n+1}& 2\ell_n -\frac{2\ell_n^2}{n+1}\\
\end{pmatrix}.$$ 

This gives
\begin{equation}\label{eqScaltoric}
    Scal(g_{FS}) = - \sum_{i,j=1}^n\frac{\partial^2 H_{ij}}{\partial x_i\partial x_j} =2n
\end{equation} 
(as desired)
and for all $j=1,\dots, n$,
\begin{equation}\label{eqLaptoric}
    \Delta_{g_{FS}}x_k = - \sum_{i,j=1}^n \frac{\partial }{\partial x_i} \left( H_{ij}\frac{\partial x_k}{\partial x_j}\right) = -  \sum_{i=1}^n\frac{\partial }{\partial x_i} H_{ik}=  2x_k
    \end{equation}
while for a function $h$ of $x_1,...,x_n$,
\begin{equation}\label{eqNormgradtoric}
    |dh|^2_{g_{FS}} =  \sum_{i,j=1}^n H_{ij} \frac{\partial h }{\partial x_i} \frac{\partial h}{\partial x_j}.
\end{equation}

Any positive ($\bbt^n$-invariant) $g_{FS}$-Killing potential can be witten as   
 $$f(x) = \langle v, x \rangle +\lambda = \sum_{i=1}^nv_i x_i +\lambda$$
 for appropriate constants $v_1,...,v_n$ and $\lambda$.      

With $g=g_1+g_{FS}$, $p=d+n+2$ and $f$ as above (pulled back to $N_1\times \bbc\bbp^n$), \eqref{weightedscal} becomes
\begin{equation}\label{eqfscaltoric}
\begin{split}
    Scal_{f,p}(g)= \frac{2n(n+1)-2d(d+1)}{n+1}f^2&(x) -4(d+n+1) f(x) \sum_{j=1}^n  x_j v_j  \\
    &- (d+n+1)(d+n+2) \sum_{i,j=1}^n \left(2 \delta_{ij}\ell_i(x) -2\frac{\ell_i(x)\ell_j(x)}{n+1}\right)v_iv_j
\end{split}
\end{equation}
 Denoting $m = \sum_{i=1}^n v_i$ and observing that $f+m-\lambda=\sum_{i=1}^n v_il_i$ we see that we can rewrite \eqref{eqfscaltoric} as
\begin{equation}\label{eqfscaltoric2}
\begin{split}
   Scal_{f,p}(g) &= \frac{2n(n+1)-2d(d+1)}{n+1}f^2(x) -4(d+n+1) f(x)( f(x) -\lambda)\\
    &\qquad  +\frac{2(d+n+1)(d+n+2)}{n+1} (f^2(x) + 2(m-\lambda) f(x) +(m-\lambda)^2) \\
    & \qquad\qquad\qquad- 2(d+n+1)(d+n+2) \sum_{i=1}^n v_i^2(1+x_i).
\end{split}
\end{equation} Note that in the last equation, the coefficient of $f^2$ on the right hand side is
$$
\begin{array}{cl}
&\frac{2n(n+1)-2d(d+1)}{n+1}-4(d+n+1) +\frac{2(d+n+1)(d+n+2)}{n+1} \\
\\
= & 2\left(\frac{n(n+1)-d(d+1)-2(d+n+1)(n+1)+(d+n+1)(d+n+2)}{n+1}\right)\\
\\
= & 2\left(\frac{n(n+1)-d(d+1)+(d+n+1)(d-n)}{n+1}\right)\\
\\
=&0
\end{array}
$$ 
and therefore $ Scal_{f,p}(g)$ is linear in $x$ and hence a Killing potential.
\end{proof}

\begin{remark}\label{polarization}
As long as the K\"ahler class $\left[\omega_1/2\pi\right]$ is a rational rescale of an integer K\"ahler class, we can do an overall rational rescale of $[\omega_1+\omega_{FS}]$ to 
get a primitive polarization of $N_1\times \bbc\bbp^n$
and hence a natural Sasaki structure above this product (via the Boothby-Wang construction). Moreover, if $(N_1, g_1, \omega_1)$ has no smooth Killing potentials at all 
(e.g. if it is negative K\"ahler-Einstein), we see that 
the Sasaki cone is determined by the set of positive smooth Killing potentials $f$ on $(\bbc\bbp^n, g_{FS}, \omega_{FS})$. 
Therefore, Proposition \ref{manytwins} tells us that the entire Sasaki cone is exhausted by extremal Sasaki metrics and they
are all twins.

We can make this more explicit when $(N_1, \omega_1,g_1)$ is a compact K\"ahler-Einstein manifold. In this case, since the scalar curvature of $(\omega_1,g_1)$
in Proposition \ref{manytwins} is given by $Scal_1 =\frac{-2d(d+1)}{n+1}$, we have that the Ricci form 
$\rho_1=\frac{Scal_1}{2d}\omega_1=-\frac{(d+1)}{n+1}\omega_1$. Since we also know that $c_1(N_1)=[\frac{\rho_1}{2\pi}] = \cali_{N_1}\gamma_1$, where $\gamma_1$ is a generator of
$H^2(N_1,\bbz)$ and $\cali_{N_1}$ is some negative integer\footnote{This is called {\em Canonical Index} in Section 3.1 of \cite{BoTo14a}, but we could also call it the {\em K\"ahler-Einstein Index}. In any case, it is meant to generalize the {\em Fano
Index}, which is specifically defined for Fano manifolds.}, we have that $[\frac{\omega_1}{2\pi}]=\frac{-(n+1)\cali_{N_1}}{d+1}\gamma_1$. Further, since $Scal_{FS}=2n$, we have that $\rho_{FS}=\omega_{FS}$ and, since the index of
$\bbc\bbp^n$ is $n+1$, we then know that $[\frac{\omega_{FS}}{2\pi}]=(n+1)\gamma_{FS}$, where $\gamma_{FS}$ is a generator of
$H^2(\bbc\bbp^n,\bbz)$. Thus $[\omega_1+\omega_{FS}]=\frac{2\pi (n+1)}{(d+1)}\left(-\cali_{N_1}\gamma_1+(d+1)\gamma_{FS}\right)$. Therefore, the corresponding primitive polarization is
$l_1 \gamma_1+l_2 \gamma_{FS}$, where $l_1=\frac{-\cali_{N_1}}{\gcd(d+1,-\cali_{N_1})}$ and $l_2=\frac{d+1}{\gcd(d+1,-\cali_{N_1})}$.
Thus if $-\cali_{N_1}$ is divisible by $(d+1)$, then $l_2=1$. Whenever $l_2=1$ is true, we know from Proposition 7.6.7 in \cite{BG05} that the corresponding Sasaki manifold is a $S^{2n+1}$-bundle over $N_1$.
This is for example the case if $N_1$ is a compact Riemann surface of genus $\gg \geq 2$. In this case
$-\cali_{N_1}=2(\gg-1)$ and $d+1=2$. Higher dimensional cases with $l_2=1$ can easily  be obtained by letting $N_1$ be the product of $d>1$ copies of a compact Riemann surface of genus $\gg\geq 3$ such that
$2(\gg-1)$ is divisible by $d+1$. [This follows from the fact that also in this case $-\cali_{N_1}=2(\gg-1)$.]
We could also let $N_1=\Sigma_{\gg_1}\times \Sigma_{\gg_2}$, where $\Sigma_{\gg_i}$ is a compact Riemann surface of genus $\gg_i=1+3k_i$, with $k_i\in \bbz^+$. Then $-\cali_{N_1}=2\gcd(\gg_1-1,\gg_2-1)=6\gcd(k_1,k_2)$.
and clearly divisible by $d+1=3$. Thus, also in this case, $l_2=1$.

Finally, we can certainly also construct examples, satisfying the assumption in Proposition \ref{manytwins}, where $(N_1, \omega_1,g_1)$ is not K\"ahler-Einstein and the corresponding primitive polarization is
$l_1 \gamma_1+l_2 \gamma_{FS}$, where $l_2=1$ and $\gamma_1$ and $\gamma_{FS}$ are generators of $H^2(N_1,\bbz)$ and $H^2(\bbc\bbp^n,\bbz)$, respectively.
As an example of this, let $N_1=\Sigma_{\gg_1}\times \Sigma_{\gg_2}$ where $\Sigma_{\gg_1}$ is a compact Riemann surface of genus $\gg_1=2k_1+1$ with $k_1\in \bbz^+$
and $\Sigma_{\gg_2}$ is a compact Riemann surface of genus $\gg_2\geq 2$.
Suppose that for $i=1,2$, the K\"ahler structure $(g_{\Sigma_{\gg_i}},\omega_{\Sigma_{\gg_i}})$ on $\Sigma_{\gg_i}$ has constant scalar curvature $4(1-\gg_i)$ and
thus $\gamma_{\Sigma_{\gg_i}}=[\omega_{\Sigma_{\gg_i}}/2\pi]$ is a generator of $H(\Sigma_{\gg_i},\bbz)$.
Then the K\"ahler structure $(\omega_1,g_1)$ on $N_1$ given by $\omega_1=(n+1)\left(k_1\omega_{\Sigma_{\gg_1}}+(\gg_2-1)\omega_{\Sigma_{\gg_2}}\right)$ 
is not K\"ahler-Einstein (since the Ricci form of $(g_1,\omega_1)$ is $\rho_1+\rho_2=-2(2k_1\omega_{\Sigma_{\gg_1}}+(\gg_2-1)\omega_{\Sigma_{\gg_2}})$) and has constant scalar curvature equal to $-12/(n+1)$, which is the required value of $\frac{-2d(d+1)}{n+1}$, 
since $d=2$. Furthermore, the K\"ahler class of the product metric is
$[\omega_1+\omega_{FS}]  =2\pi (n+1)\left(k_1 \gamma_{\Sigma_{\gg_1}}+(\gg_2-1)\gamma_{\Sigma_{\gg_2}}+\gamma_{FS}\right)$. 
Thus, the corresponding primitive polarization is
$l_1 \gamma_1+l_2 \gamma_{FS}$ where $l_1=\gcd(k_1,(\gg_2-1))$, $\gamma_1=\frac{1}{\gcd(k_1,(\gg_2-1))}\left(k_1 \gamma_{\Sigma_{\gg_1}}+(\gg_2-1)\gamma_{\Sigma_{\gg_2}}\right)$, and $l_2=1$ (as desired).

\end{remark}

\begin{remark}\label{nocscStwinshere}
Assuming the assumptions of Proposition \ref{manytwins} and assume that we are in the situation of Remark \ref{polarization}. 
If $d\geq n$ then $Scal_1 =\frac{-2d(d+1)}{n+1}\leq -2n=-Scal_{FS}$ and hence when $f$ is constant, 
the product metric on 
$N_1\times \bbc\bbp^n$ 
has non-positive constant scalar curvature $Scal_{f,p}(g)$. In that case, it follows as a special case of the uniqueness result in \cite{JeLe87} that no other choice of $f=v(x_1+\cdots+x_l)+\lambda$ would give $Scal_{f,p}(g)$ constant.

More generally, for any choice of $d,n\in \bbz^+$, we have that
\begin{equation}
\begin{split}
   Scal_{f,p}(g) &= 4(d+n+1)\left(\frac{\lambda(n+1) +(d+n+2)(m-\lambda)}{n+1}\right)f\\
    &\qquad  +\frac{2(d+n+1)(d+n+2)}{n+1} \left((m-\lambda)^2 -(n+1) \sum_{i=1}^n v_i^2(1+x_i)\right)\\
\end{split}
\end{equation}
From the statement above \eqref{cscS}, we know that we have a cscS solution if and only if $Scal_{f,p}(g)$ is a constant times $f$. Clearly this is equivalent to the existence of a constant $k$ such that
$$(m-\lambda)^2 -(n+1) \sum_{i=1}^n v_i^2(1+x_i)=kf.$$
Recall that $m= \sum_{i=1}^n v_i$ and $f=\sum_{i=1}^nv_i x_i +\lambda$. With this, the condition above becomes
$$( \sum_{i=1}^n v_i-\lambda)^2 -(n+1) \sum_{i=1}^n v_i^2(1+x_i)=k(\sum_{i=1}^nv_i x_i +\lambda),$$
for all $x_i$ in their appropriate domain.
Comparing the constant and the $x_i$ terms on each side of the equation, we get a system of $n+1$ equations in $v_1,\dots,v_n$ and $k$:
\begin{equation}
\begin{split}
(\sum_{i=1}^n v_i-\lambda)^2 -(n+1) \sum_{i=1}^n v_i^2 & = k\lambda\\
 -(n+1)  v_i^2 & =  k v_i .\\
\end{split}
\end{equation}
The trivial solution to this system is $v_i=0$ for all $i=1,...,n$ and $k=\lambda$. Assuming we have a non-trivial solution, then, without loss of generality, there is some integer $1\leq l <n$ such that $v_1=v_2=\cdots=v_l=v\neq 0$, $v_{l+1}=\cdots = v_n=0$, $k=-(n+1)v$,
and
\begin{equation}\label{nontrivial}
(l v-\lambda)^2-(n+1)lv^2=-\lambda(n+1) v.
\end{equation}
Equation \eqref{nontrivial} simplifies to
$$\left((n+1)l-l^2\right)v^2+\lambda\left(2l-(n+1)\right) v-\lambda^2=0,$$
which has solutions $v=\frac{\lambda}{l}$ and $v=\frac{-\lambda}{(n-l)+1}$.
Since we need $f=v(x_1+\cdots+x_l)+\lambda$ to be strictly positive over the closed domain given by $x_i\geq -1$ for $i=1,...,n$ and $x_1+\cdots+x_n\leq 1$,
it easy to see that none of these options will work. In conclusion, the only extremal Sasaki metric in Proposition \ref{manytwins} that is cscS is the regular one ($f$ constant) with K\"ahler quotient equal to the original product metric.
\end{remark}

\begin{example}\label{n=d=1}
Let us assume that 
$n=d=1$ in Proposition \ref{manytwins}. That is, we are considering $\Sigma_\gg\times \bbc\bbp^1$ where 
$\Sigma_\gg$ has genus $\gg \geq 2$ and the assumption made in the proposition is that $g_{\Sigma_\gg}$ is scaled such that $Scal_{\Sigma_\gg}=-2$. Note that since the
the K\"ahler-Einstein index of ${\Sigma_\gg}$ is $2(1-\gg)$ it is possible to arrange this by choosing $\omega_{\Sigma_\gg}$ such that
$\left[ \omega_{\Sigma_\gg}/2\pi\right]$ is  $2(\gg-1)$ times a generator of $H^2({\Sigma_\gg},\bbz)$.

By rescaling, $Scal_{\Sigma_\gg}$ could take any negative real value and we will now consider \eqref{weightedscal} in the case where
$Scal_{\Sigma_\gg}=k\in \bbq\cap(-\infty,0)$ and
the K\"ahler $g_{\bbc\bbp^1}$ metric on $\bbc\bbp^1$ is given by the simplified toric set-up:
\begin{equation}
g_{\bbc\bbp^1}=\frac {d\gz^2}
{H(\gz)}+H(\gz)d\phi^2,\quad
\omega_{\bbc\bbp^1} = d\gz \wedge d\phi,
\end{equation}
where $-1<\gz<1$, $0<\phi\leq 2\pi$, and
\begin{align}
(i)~~ H(\gz) > 0, \quad -1 < \gz <1,\quad
(ii) ~~H(\pm 1) = 0,\quad
(iii)~~ H'(\pm 1) = \mp 2.
\end{align}
We may think of $\gz:\bbc\bbp^1 \rightarrow [-1,1]$ as a moment map of a choice of a $S^1$ action on $\bbc\bbp^1$ and $\omega_{\bbc\bbp^1}$.
(It is the same as $x_1$ in the proof above with $n=1$.)
The scalar curvature of $g_{\bbc\bbp^1}$ is given by 
$Scal=-H''(\gz)$  and for $f=f(\gz)$ we have that
$\Delta_{g_{\bbc\bbp^1}} f =-\frac{d}{d\gz}\left[H(\gz)f'(\gz)\right]$
and $|d\gz|^2_{g_{\bbc\bbp^1}}=H(\gz)$. 
If we choose $H(\gz)=1-\gz^2$, we are back to the case of Proposition \ref{manytwins}.
Note that $\left[\omega_{\bbc\bbp^1}/2\pi\right]$ and $\left[ \omega_{\Sigma_\gg}2\pi\right]$ are both rational
cohomology classes, so Remark \ref{polarization} applies here to the class $\left[\omega_{\Sigma_\gg}+\omega_{\bbc\bbp^1}\right]$
of the product K\"ahler form on $\Sigma\times \bbc\bbp^1$.
By considering different value of $Scal_{\Sigma_\gg}=k\in \bbq\cap(-\infty,0)$, we are implicitly considering all the possible polarization choices for 
${\Sigma_\gg}\times \bbc\bbp^1$. 

We are now ready to write up \eqref{weightedscal} for $g=g_{\Sigma_\gg}+g_{\bbc\bbp^1}$, $p=1+1+2=4$, and $f= c\gz+1$ where $-1<c<1$ and $\gz$ is the pullback from $\bbc\bbp^1$ to ${\Sigma_\gg}\times \bbc\bbp^1$:
\begin{equation}
Scal_{f,4}(g)= (1+c\gz)^2 \left(k-H''(\gz)\right)  +6c (1+c\gz)H'(\gz) - 12c^2H(\gz).
\end{equation}
As expected, when $k=-2$ and $H(\gz)=1-\gz^2$, $Scal_{f,4}(g)$ is a linear function in $\gz$ for any choice of $c\in (-1,1)$.

In general, we are looking for solutions to 
$$ (1+c\gz)^2 \left(k-H''(\gz)\right)  +6c (1+c\gz)H'(\gz) - 12c^2H(\gz)=A\gz+B$$
where $A,B$ are real constants.
This can be rewritten as
$$-(1+c\gz)^2 H''(\gz) +6c (1+c\gz)H'(\gz) - 12c^2H(\gz)=-k(1+c\gz)^2+A\gz+B$$
or
$$\frac{d^2}{d\gz^2}\left[\frac{H(\gz)}{(c\gz+1)^3}\right]=\left(k(1+c\gz)^2-A\gz-B\right)(c\gz+1)^{-5},$$
which, due to the endpoint conditions on $H(\gz)$, determines ($A$, $B$ and) $H(\gz)$ in terms of $k$, and $c$. 
We have
$$A=\frac{6 c \left(c^2 k-k+4\right)}{c^2-3},\quad B = \frac{3 \left(c^4 k-2 c^4+12 c^2-k-2\right)}{c^2-3}$$
and
$$H(\gz)=H_{k,c}(\gz)=\frac{(1-\gz^2) \left(4 \left(3-c^2\right)+(k+2) c^2(1- \gz^2)\right)}{4 \left(3-c^2\right)}$$
which satisfies the required $H(\gz) > 0$ for $-1 < \gz <1$ if and only if $12-(2-k)c^2>0$.
We assume this condition (which is non-trivial when $k<-10$) from now on.

Note that $A\gz+B$ is only a constant multiple of $c\gz+1$ (meaning the  Reeb vector field associated to $c$ is cscS) when $c=0$ and $H(\gz)=1-\gz^2$.

It is easy to see that $H_{k,c_1}=H_{k,c_2}$ if and only if $c_2=-c_1$ or $k=-2$. Thus, we either have that everything is twinning (when $k=-2$ as in Proposition \ref{manytwins}) or the twins appear 
as genuine pairs $\pm c$ (when $k\neq -2$) which are related by the Weyl group.

\end{example}

\subsection{Other weights}\label{otherweights}

Note that beyond the Sasakian realm, Proposition \ref{manytwins} has a generalization that can be proved in the same manner as the proof of Proposition \ref{manytwins}:

\begin{proposition}\label{generalmanytwins}
Let $p \in \bbr$,
let $(N_1, g_1, \omega_1)$ be a compact K\"ahler manifold of complex dimension $d$ with constant scalar curvature $Scal_1 =\frac{2(2-(p-n))((p-n)-1)}{n+1}$, and
let $(N_2, g_2, \omega_2)=(\bbc\bbp^n, g_{FS}, \omega_{FS})$ denote the complex projective space of complex dimension $n$ with the standard Fubini-Study metric of constant scalar curvature 
$Scal_{FS}=2n$.
Let $f$ be the pullback to $N_1\times \bbc\bbp^n$ of any positive smooth $\bbt^n$-invariant Killing potential $f$ on $(\bbc\bbp^n, g_{FS}, \omega_{FS})$. 
Then $(N_1\times \bbc\bbp^n, g_1+g_{FS},\omega_1+\omega_{FS})$
is an $(f,p)$-extremal K\"ahler structure.
\end{proposition}

The expression for $Scal_1$ as a function of $(p-n)$ is a concave down parabola with a maximum value of $\frac{1}{2(n+1)}$.
In particular, for $p=n-d+1$ and $p=d+n+2$ we have the same value $Scal_1=\frac{-2d(d+1)}{n+1}$. Thus Proposition \ref{manytwins} has an embellished version:
\begin{proposition}\label{manytwinsbonus}
Let $(N_1, g_1, \omega_1)$ be a compact K\"ahler manifold of complex dimension $d$ with constant scalar curvature $Scal_1 =\frac{-2d(d+1)}{n+1}$ and
let $(N_2, g_2, \omega_2)=(\bbc\bbp^n, g_{FS}, \omega_{FS})$ denote the complex projective space of complex dimension $n$ with the standard Fubini-Study metric of constant scalar curvature 
$Scal_{FS}=2n$.
Let $f$ be the pullback to $N_1\times \bbc\bbp^n$ of any positive smooth $\bbt^n$-invariant Killing potential $f$ on $(\bbc\bbp^n, g_{FS}, \omega_{FS})$. 
Then the K\"ahler structure $(N_1\times \bbc\bbp^n, g_1+g_{FS},\omega_1+\omega_{FS})$
is $(f,n-d+1)$-extremal  as well as
$(f,d+n+2)$-extremal.
\end{proposition}

\begin{remark}
For $p=n+2$ and $p=n+1$ in Proposition \ref{generalmanytwins} we have $Scal_1=0$.
We can formally ignore the $(N_1, g_1, \omega_1)$ piece (by setting $d=0$ in Proposition \ref{manytwinsbonus}) and consider
$p=n+2$ as the $\bbc\bbp^n$ case discussed right below Definition \ref{twinsSasaki}.  
This then produces the observation that the Fubini-Study K\"ahler structure $(\bbc\bbp^n, g_{FS}, \omega_{FS})$ is
 $(f,n+1)$-extremal as well as  $(f,n+2)$-extremal
for any choice of a positive smooth Killing potential $f$ on $(\bbc\bbp^n, g_{FS}, \omega_{FS})$.

This ``two-weights'' observation is analogues, but not identical, to an observation made in \cite{ApJuLa21} where it is discovered that on a smooth Fano manifold $X$, with a non-trivial compact torus $\bbt \subset Aut(X)$, any non-trivial $v$-soliton
 in $2\pi c_1(X)$ would automatically be $(v,w)$-cscK (in the sense of Definition 4 of \cite{Lah19}) as well as $(\tilde{v},\tilde{w})$-cscK, where the weights $(\tilde{v},\tilde{w})$ are different from $(v,w)$.
 
Note that $(v,w)$-extremal as well as  $(v,w)$-cscK in the sense of Lahdili (See the introduction of  \cite{Lah19} as well as Section 1.1 of \cite{ApJuLa21}) is much more general than the notion of $(f,p)$-extremal (defined in \cite{ApCa18}) that we are using in the present paper. The latter is a special case of the former where $v=f^{-(p-1)}$ and $w=f^{-(p+1)}$.
\end{remark}

\newcommand{\etalchar}[1]{$^{#1}$}
\def\cprime{$'$} \def\cprime{$'$} \def\cprime{$'$} \def\cprime{$'$}
  \def\cprime{$'$} \def\cprime{$'$} \def\cprime{$'$} \def\cdprime{$''$}
  \def\cprime{$'$} \def\cprime{$'$} \def\cprime{$'$} \def\cprime{$'$}
\providecommand{\bysame}{\leavevmode\hbox to3em{\hrulefill}\thinspace}
\providecommand{\MR}{\relax\ifhmode\unskip\space\fi MR }
\providecommand{\MRhref}[2]{%
  \href{http://www.ams.org/mathscinet-getitem?mr=#1}{#2}
}
\providecommand{\href}[2]{#2}

\end{document}